\documentclass[11pt,a4paper,reqno]{amsart}

\usepackage[T1]{fontenc}
\usepackage[english]{babel}

\usepackage{mathtools}
\usepackage{amssymb}
\usepackage{libertinus}
\usepackage[cal=boondoxo,bb=ams]{mathalfa}
\usepackage{microtype}

\allowdisplaybreaks[2]
\usepackage[
a4paper,
left=30mm,
right=30mm,
top=27mm,
bottom=30mm,
headsep=8mm,
footskip=13mm,
heightrounded
]{geometry}

\usepackage{graphicx}
\usepackage{xcolor}
\usepackage{tikz}
\usepackage{pgfplots}

\usepgfplotslibrary{fillbetween}
\pgfplotsset{compat=1.18}

\usetikzlibrary{
	patterns,
	patterns.meta,
	arrows.meta,
	positioning,
	matrix
}

\usepackage{enumitem}

\setlist[itemize]{
	leftmargin=2.1em,
	itemsep=0.25em,
	topsep=0.45em,
	parsep=0pt
}

\setlist[enumerate]{
	leftmargin=2.1em,
	itemsep=0.25em,
	topsep=0.45em,
	parsep=0pt
}

\numberwithin{equation}{section}

\theoremstyle{plain}
\newtheorem{theorem}{Theorem}[section]
\newtheorem{lemma}[theorem]{Lemma}

\newtheorem{coro}[theorem]{Corollary}

\theoremstyle{definition}

\theoremstyle{remark}
\newtheorem{remark}[theorem]{Remark}

\makeatletter

\def\@settitle{%
	\begin{center}
		\vspace*{-0.4em}
		{\normalfont\bfseries
			\fontsize{17}{21}\selectfont
			\@title\par}
		\vspace{0.7em}
	\end{center}
}

\renewcommand\section{%
	\@startsection{section}{1}{\z@}%
	{1.5\baselineskip plus 0.2\baselineskip minus 0.1\baselineskip}%
	{0.65\baselineskip}%
	{\normalfont\Large\bfseries}%
}

\renewcommand\subsection{%
	\@startsection{subsection}{2}{\z@}%
	{1.15\baselineskip plus 0.2\baselineskip minus 0.1\baselineskip}%
	{0.45\baselineskip}%
	{\normalfont\large\bfseries}%
}

\renewcommand\subsubsection{%
	\@startsection{subsubsection}{3}{\z@}%
	{0.9\baselineskip plus 0.15\baselineskip minus 0.1\baselineskip}%
	{0.35\baselineskip}%
	{\normalfont\normalsize\bfseries}%
}

\makeatother

\usepackage{etoolbox}

\makeatletter

\patchcmd{\@setauthors}
{\centering\footnotesize}
{\centering\normalsize}
{}
{\PackageWarning{Nakao-template}
	{Failed to patch the author font size}}

\patchcmd{\@setauthors}
{\MakeUppercase{\authors}}
{\authors}
{}
{\PackageWarning{Nakao-template}
	{Failed to patch the author capitalization}}

\patchcmd{\maketitle}
{\uppercasenonmath\shorttitle}
{}
{}
{\PackageWarning{Nakao-template}
	{Failed to patch the running-title formatting}}

\patchcmd{\maketitle}
{\@nx\MakeUppercase{\the\toks@}}
{\the\toks@}
{}
{\PackageWarning{Nakao-template}
	{Failed to patch the running-author formatting}}

\makeatother

\usepackage{hyperref}

\hypersetup{
	hidelinks,
	pdfencoding=auto
}

\usepackage[nameinlink,capitalise,noabbrev]{cleveref}

\newcommand{\ml}{\mathcal}
\newcommand{\mb}{\mathbb}

\DeclareMathOperator{\lin}{lin}
\DeclareMathOperator{\nlin}{nlin}

\title[Small data global in-time existence for Nakao's problem in two and three space dimensions]
{Small data global in-time existence for Nakao's problem in two and three space dimensions}

\author[W. Chen]{Wenhui Chen}

\address{
	School of Mathematics and Information Science,
	Guangzhou University,
	Guangzhou 510006,
	P. R. China
}

\email{wenhui.chen.math@gmail.com}

\date{}

\begin{document}
		
\begin{abstract}
We study Nakao's problem in two and three space dimensions, a weakly coupled system consisting of a semilinear damped wave equation and a semilinear undamped wave equation. We establish the global in-time existence and uniqueness of small data mild Sobolev solutions in new admissible ranges, together with time-dependent estimates matching those for the corresponding linearized problems. The proof combines diffusion-type $L^m-L^r$ estimates for the damped component with the Poisson and Kirchhoff formulas for the undamped component. Dimension-dependent solution spaces are introduced to incorporate the weaker wave decay and logarithmic $L^2$ growth in two space dimensions. In three space dimensions, a two-level fixed point argument avoids an artificial restriction on $p$ by establishing the self-map property in a strong space and the contraction in a weaker metric. Furthermore, the undamped component scatters to a free wave in $H^2\times H^1$.
\end{abstract}

\maketitle

\section{Introduction}
In this paper, we study Nakao's problem in two and three space dimensions, namely, the weakly coupled Cauchy problem
\begin{align}\label{Nakao-Problem}
	\begin{cases}
		u_{tt}-\Delta u+u_t=|v|^p,&x\in\mb{R}^n,\ t>0,\\
		v_{tt}-\Delta v=|u|^q,&x\in\mb{R}^n,\ t>0,\\
		(u,u_t)(0,x)=(u_0,u_1)(x),&x\in\mb{R}^n,\\
		(v,v_t)(0,x)=(v_0,v_1)(x),&x\in\mb{R}^n,
	\end{cases}
\end{align}
where $n\in\{2,3\}$ and $p,q>1$. This system has a mixed dissipative-dispersive structure. The first component is governed by a damped wave equation and exhibits diffusion-type decay, in accordance with the diffusion phenomenon for damped waves, for example, \cite{Nishihara=2003}. By contrast, the second component retains the dispersive character of a free wave. Its representation is governed by the Poisson formula in two space dimensions and by the Kirchhoff formula in three space dimensions. Thus, the system couples two components with fundamentally different large time behavior. One is diffusion-like, whereas the other remains genuinely wave-like. This asymmetry precludes a direct application of the standard methods developed for either weakly coupled wave-wave systems or weakly coupled damped-damped wave systems.

We now recall some background related to Nakao's problem. Because \eqref{Nakao-Problem} contains one damped component and one undamped component, it is natural to compare it with two benchmark weakly coupled systems, namely, the semilinear wave-wave system and the semilinear damped-damped wave system. Particularly, let us recall their critical curves in the $p-q$ plane, which serve as the thresholds between the small data global in-time existence range and the finite-time blow-up range for the corresponding benchmark problems.
\begin{itemize}
	\item The first benchmark problem is the weakly coupled system of semilinear wave equations
	\begin{align}\label{Wave-System}
		\begin{cases}
			u_{tt}-\Delta u=|v|^p,&x\in\mb{R}^n,\ t>0,\\
			v_{tt}-\Delta v=|u|^q,&x\in\mb{R}^n,\ t>0,\\
			(u,u_t)(0,x)=(u_0,u_1)(x),&x\in\mb{R}^n,\\
			(v,v_t)(0,x)=(v_0,v_1)(x),&x\in\mb{R}^n,
		\end{cases}
	\end{align}
	where $n\in\mb{N}$ and $p,q>1$. This model corresponds to the purely wave-like counterpart of \eqref{Nakao-Problem} and has been studied extensively (see
	\cite{Delsanto=1997,DelSanto-Georgiev-Mitidieri=1997,DelSanto-Mitidieri=1998,Kubo-Ohta=1999,Agemi-Kurokawa-Takamura=2000,Kurokawa-Takamura=2003,Kurokawa=2005,Georgiev-Takamura-Zhou=2006,Kurokawa-Takamura-Wakasa=2012,Ikeda-Sobajima-Wakasa=2019}
	and the references therein). Its critical curve is described by the quantity
	\begin{align*}
		\Gamma_{\mathrm{Wave},n}(p,q):=\max\left\{\frac{p+2+q^{-1}}{pq-1},\frac{q+2+p^{-1}}{pq-1}\right\}-\frac{n-1}{2},
	\end{align*}
	which is of Strauss-type. Indeed, in the symmetric case $p=q$, it reduces to the Strauss critical exponent (see, for example, \cite{John=1979,Strauss=1981}) for the semilinear wave equation.
	\item The second benchmark problem is the weakly coupled system of semilinear damped wave equations
	\begin{align}\label{Damped-Wave-System}
		\begin{cases}
			u_{tt}-\Delta u+u_t=|v|^p,&x\in\mb{R}^n,\ t>0,\\
			v_{tt}-\Delta v+v_t=|u|^q,&x\in\mb{R}^n,\ t>0,\\
			(u,u_t)(0,x)=(u_0,u_1)(x),&x\in\mb{R}^n,\\
			(v,v_t)(0,x)=(v_0,v_1)(x),&x\in\mb{R}^n,
		\end{cases}
	\end{align}
	where $n\in\mb{N}$ and $p,q>1$. This model corresponds to the fully damped counterpart of \eqref{Nakao-Problem} and reflects a diffusion-like interaction in both components. It has been investigated recently (see
	\cite{Sun-Wang=2007,Narazaki=2009,Nishihara=2012,Nishihara-Wakasugi=2014,Nishihara-Wakasugi=2015,Chen-Dao=2023}
	and the references therein). Its critical curve is given by the quantity
	\begin{align*}
		\Gamma_{\mathrm{DWave},n}(p,q):=\max\left\{\frac{p+1}{pq-1},\frac{q+1}{pq-1}\right\}-\frac{n}{2},
	\end{align*}
	which is of Fujita-type. Indeed, in the symmetric case $p=q$, it reduces to the Fujita critical exponent (see, for example, \cite{Fujita=1966,Li-Zhou=1995}) for the semilinear heat equation. This is consistent with the diffusion phenomenon for the semilinear damped wave equation.
\end{itemize}
In other words, for
$\Gamma_n(p,q)\in\{\Gamma_{\mathrm{Wave},n}(p,q),\Gamma_{\mathrm{DWave},n}(p,q)\}$,
the condition $\Gamma_n(p,q)<0$ describes the expected small data global in-time existence range, whereas $\Gamma_n(p,q)\geqslant0$ gives, in general, the corresponding finite-time blow-up range, possibly under the usual sign, non-triviality or compact support assumptions on the initial data.

Motivated by the two benchmark systems above, Mitsuhiro Nakao proposed the problem of determining the critical curve for the mixed weakly coupled system, now known as Nakao's problem, in the whole space $\mb{R}^n$. He first studied related coupled systems of wave and damped wave equations in bounded domains (see, for example, \cite{Nakao=2016,Nakao=2018}). In contrast to the two standard systems recalled above, Nakao's problem contains only one frictional damping term, namely, $+u_t$ in the first equation. Hence, its critical curve cannot be read off directly from either the Strauss-type curve for the wave system \eqref{Wave-System} or the Fujita-type curve for the damped wave system \eqref{Damped-Wave-System}. Instead, it is expected to describe a genuinely mixed threshold, reflecting how the diffusion-like behavior of the damped component $u$ and the wave-like propagation of the undamped component $v$ are weakly coupled through the nonlinear source terms.

Let us now summarize the known results for Nakao's problem \eqref{Nakao-Problem} in the whole space $\mb{R}^n$. In the usual energy framework, the restrictions $1<p,q<\infty$ for $n\leqslant2$ and $1<p,q\leqslant\frac{n}{n-2}$ for $n\geqslant3$, which follow from an application of the Gagliardo-Nirenberg inequality, are imposed to ensure local in-time existence. By using the classical test function method introduced by \cite{Zhang=2001}, a blow-up result for weak solutions was proved in \cite{Wakasugi=2017} if $\Gamma_{\mathrm{Wak},n}(p,q)\geqslant0$, where
\begin{align*}
	\Gamma_{\mathrm{Wak},n}(p,q)&:=\max\left\{\frac{q+1}{pq-1}-\frac{n}{2},\frac{q+2}{pq-1}-(n-1),\frac{p+1}{pq-1}-\frac{n}{2}\right\}\notag\\[0.3em]
	&\ =\max\left\{\frac{q+2}{pq-1}-(n-1),\Gamma_{\mathrm{DWave},n}(p,q)\right\}.
\end{align*}
We point out that the one-dimensional case is already completely
settled by this blow-up result. Actually, for $n=1$, one has
\begin{align*}
	\Gamma_{\mathrm{Wak},1}(p,q)
	\geqslant
	\frac{q+2}{pq-1}>0\ \ \mbox{for every}\ \ p,q>1.
\end{align*}
Hence, under the assumptions imposed in \cite{Wakasugi=2017}, finite-time blow-up occurs throughout the entire exponent range $p,q>1$. In this sense, the blow-up result is optimal with respect to the nonlinear powers in one space dimension, and the genuinely open global in-time existence problem begins in dimensions $n\geqslant 2$.
This condition via $\Gamma_{\mathrm{Wak},n}(p,q)$ has a diffusion-like character and is related to the Fujita-type threshold for the weakly coupled damped wave system \eqref{Damped-Wave-System}. Subsequently, a wave-like blow-up result for Nakao's problem was obtained in \cite{Chen-Reissig=2021} by using an iteration method combined with the slicing procedure. This procedure was originally introduced in \cite{Agemi-Kurokawa-Takamura=2000} and was later further developed in \cite{Chen-Palmieri=2020} to deal with unbounded exponential multipliers. More precisely, every non-trivial energy solution blows up in finite time if $\Gamma_{\mathrm{CR},n}(p,q)>0$, where
\begin{align*}
	\Gamma_{\mathrm{CR},n}(p,q):=\max\left\{\frac{q+2}{pq-1}-(n-1),\frac{2p+1}{pq-1}-n,\frac{2+p^{-1}}{pq-1}-\frac{n-1}{2}\right\}.
\end{align*}
More recently, the test function method developed in \cite{Ikeda-Sobajima-Wakasa=2019} was applied in \cite{Kita-Kusaba=2022} to demonstrate finite-time blow-up for positive and compactly supported initial data whenever
\begin{align*}
	\Gamma_{\mathrm{Nakao},n}(p,q):=\max\big\{\Gamma_{\mathrm{Wak},n}(p,q),\Gamma_{\mathrm{CR},n}(p,q)\big\}>0.
\end{align*}
Until very recently, the small data global in-time existence for Nakao's problem in the whole space had remained open. The authors of \cite{Georgiev-Kita=2026} have recently established global in-time existence in three space dimensions by developing weighted pointwise estimates for the Duhamel term of the damped wave equation. Their estimates combine the parabolic weight associated with the diffusion phenomenon and a wave-cone weight inherited from the undamped component. By working in weighted $L^\infty$ spaces, they obtain global in-time mild solutions with either the standard or weaker pointwise decay rates. Combining the exponent conditions in their global in-time existence corollaries, their three-dimensional admissible region can be written as
\begin{align*}
	\Omega_{\mathrm{GK},3}:={}&\left\{(p,q)\in\mb{R}^2:\ 1<p\leqslant3\ \ \mbox{and}\ \ q>\frac4p+\frac2{p^2}\right\}\\
	&\quad\cup\left\{(p,q)\in\mb{R}^2:\ p>3 \ \ \mbox{and}\ \ q>1+\frac5{3p}\right\}.
\end{align*}
At the level of the power exponents $p$ and $q$, the three-dimensional region obtained in the present paper is contained in $\Omega_{\mathrm{GK},3}$.

The two approaches, however, are formulated in substantially different functional settings and provide different information on the solutions. The result in \cite{Georgiev-Kita=2026} assumes polynomial pointwise decay of the initial data and controls the solutions in weighted space-time $L^\infty$ spaces. In contrast, the present paper works in a Sobolev-energy framework based on diffusion-type $L^m-L^r$ estimates and spatial integrability assumptions. It yields mild Sobolev solutions with $H^2\times H^1$ regularity, time-dependent estimates matching the corresponding linearized problems, and asymptotic free wave behavior for the undamped component. In particular, the data classes and the resulting solution theories are not directly comparable.

The main purpose of this paper is therefore twofold. First, we establish small data global in-time existence for Nakao's problem in two space dimensions, where, to the best of the author's knowledge, no previous global in-time existence result in the whole space is available. Second, we develop a dimension-dependent Sobolev framework that applies simultaneously in two and three space dimensions and complements the recent weighted pointwise approach in the three-dimensional case \cite{Georgiev-Kita=2026}. Our admissible ranges are
\begin{align}\label{Range-2D}
	\Omega_{\mathrm{Nakao},2}&:=\left\{(p,q)\in\mb{R}^2:\  p>\max\left\{3,2+\frac4q\right\}\ \ \mbox{and}\ \ q>2\right\},\\[0.3em]
	\label{Range-3D}
	\Omega_{\mathrm{Nakao},3}&:=	\left\{(p,q)\in\mb{R}^2:\ p>\max\left\{2,1+\frac{10}{3q}\right\}\ \ \mbox{and}\ \ q>2\right\}.
\end{align}

The proof is organized around two main difficulties.
\begin{description}
	\item[Difficulty I] We have to place the damped and undamped components in compatible time-weighted spaces, despite the different linear mechanisms governing their large time behavior. The damped component is controlled by diffusion-type $L^m-L^{r_n}$ estimates with suitable $m$ and $r_n$ that will be determined later, whereas the undamped component is estimated through the Poisson formula in two space dimensions and the Kirchhoff formula in three space dimensions. In the two-dimensional case, the weaker wave decay and the optimal logarithmic $L^2$ growth require a dimension-dependent solution space. Moreover, we choose
	$r_2=\min\{q,4\}$ and $r_3=\min\{q,\frac{10}{3}\}$ rather than always taking $r_n=q$. These choices separate the Sobolev regularity required by the damped wave estimate from the exponent appearing in the nonlinearity $|u|^q$ and allow $q$ to be arbitrarily large in both dimensions.
	\item[Difficulty II] The strong self-map estimate requires second-order derivative bounds for $|v|^p$, while a contraction estimate in the same norm would involve the difference of the second-order derivatives of the nonlinearities. For $2<p<3$, this would impose the artificial restriction $p\geqslant3$. We therefore establish the self-map property in a strong solution space and justify the contraction in a weaker metric involving only first-order derivative nonlinear differences. This two-level fixed point argument allows us to treat the portion
	$2<p<3$ of the three-dimensional Sobolev range without imposing the
	artificial restriction $p\geqslant3$. Although the condition $p>3$ would allow one to establish the contraction directly in the strong norm when $n=2$, we use the same
	weaker metric in both dimensions $n\in\{2,3\}$ in order to retain a unified argument. The loss of one spatial derivative in the contraction metric is essential only in the three-dimensional case.
\end{description}

\paragraph{Notation.}
Throughout this paper, the constants $C$ and $c$ are positive and may change from line to line. We write $f\lesssim g$ if there exists a constant $C>0$, independent of $t$, $T$ and the initial data, such that $f\leqslant Cg$. We write $f\approx g$ if both $f\lesssim g$ and $g\lesssim f$ hold. The symbol $\ast_{(x)}$ denotes convolution with respect to the spatial variable. For $s\geqslant0$ and $1<r<\infty$, the notation $W^{s,r}$ stands for the Bessel-potential space equipped with the norm
\begin{align*}
	\|f\|_{W^{s,r}}:=\left\|(1+|\nabla|^2)^{\frac{s}{2}} f\right\|_{L^r}.
\end{align*}
For an integer $k\geqslant0$ and $r\in\{1,\infty\}$, the notation $W^{k,r}$ denotes the usual Sobolev space equipped with norm 
\begin{align*}
	\|f\|_{W^{k,r}}:=\sum_{|\alpha|\leqslant k}\|\partial_x^\alpha f\|_{L^r}.
\end{align*} 
Especially, we write $H^s:=W^{s,2}$. We denote by $|\nabla|$ the Fourier multiplier with symbol $|\xi|$. Finally, we put $\langle t\rangle:=1+t$.

\section{Main results}

\subsection{Global in-time existence theorem}

 We now state the main result of this paper. The parameter $m$ is an auxiliary integrability exponent used in the $L^m-L^r$ estimates for the damped wave component $u$. For notational convenience, we introduce the parameter
\begin{align}\label{Dimension-Parameters}
	r_n:=
	\begin{cases}
		\min\{q,4\}&\mbox{if}\ \ n=2,\\
		\min\left\{q,\frac{10}{3}\right\}&\mbox{if}\ \ n=3,
	\end{cases}
\end{align}
and the time-dependent function
\begin{align}\label{Dimension-Parameters-2}
	\Lambda_n(t):=
	\begin{cases}
		\sqrt{\ln(\mathrm{e}+t)}&\mbox{if}\ \ n=2,\\
		1&\mbox{if}\ \ n=3.
	\end{cases}
\end{align}
The assumptions on $(u_0,u_1)$ are chosen to match the Sobolev regularity required by the linear damped wave equation, whereas those on $(v_0,v_1)$ are imposed in order to obtain the Poisson- or Kirchhoff-type $L^\infty$ estimate and the corresponding energy bound for the undamped wave component. Here and below, a mild Sobolev solution means a pair belonging to the stated Sobolev class and satisfying the Duhamel formulas introduced in Section~\ref{Subsection-Functional-Setting}.

\begin{theorem}[Global in-time existence]\label{Thm-01}
Let $n\in\{2,3\}$, and let $r_n$ and $\Lambda_n$ be as in \eqref{Dimension-Parameters} and \eqref{Dimension-Parameters-2}, respectively.
	\begin{itemize}
		\item If $n=2$, let $(p,q)\in\Omega_{\mathrm{Nakao},2}$ defined in \eqref{Range-2D} and fix a parameter $m$ satisfying
		\begin{align}\label{Restriction-m-2D}
			\max\left\{1,\frac{2}{p-2}\right\}<m<\min\left\{2,\frac{q r_2}{q+r_2}\right\}.
		\end{align}
		\item If $n=3$, let $(p,q)\in\Omega_{\mathrm{Nakao},3}$ defined in \eqref{Range-3D} and fix a parameter $m$ satisfying
		\begin{align}\label{Restriction-m-3D}
			\max\left\{1,\frac{2}{p-1},\frac{3r_3}{2r_3+3}\right\}<m<\min\left\{2,\frac{3qr_3}{3q+2r_3}\right\}.
		\end{align}
	\end{itemize}
	Suppose that
	\begin{align*}
		(u_0,u_1)\in\ml{A}_n&:=\big(W^{2,m}\cap W^{3-\frac{2}{r_n},r_n}\cap H^2\big)\times\big(W^{2,m}\cap W^{2-\frac{2}{r_n},r_n}\cap H^1\big),\\[0.3em]
		(v_0,v_1)\in\ml{B}_n&:=\big(W^{2,1}\cap W^{1,\infty}\cap H^2\big)\times\big(W^{1,1}\cap L^\infty\cap H^1\big).
	\end{align*}
	Then, there exists a constant $\varepsilon_0>0$ such that if
	\begin{align*}
		\|(u_0,u_1)\|_{\ml{A}_n}+\|(v_0,v_1)\|_{\ml{B}_n}\leqslant\varepsilon_0,
	\end{align*}
	then Nakao's problem \eqref{Nakao-Problem} admits a global in-time mild Sobolev solution $(u,v)$ satisfying
	\begin{align*}
		\begin{cases}
			\displaystyle{u\in\ml{C}\big([0,\infty),W^{2,r_n}\cap H^2\big)\cap\ml{C}^1\big([0,\infty),H^1\big)},\\[0.5em]
			\displaystyle{v\in\ml{C}\big([0,\infty),L^\infty\cap H^2\big)\cap\ml{C}^1\big([0,\infty),H^1\big)}.
		\end{cases}
	\end{align*}
	This solution is unique in the class of sufficiently small global in-time mild Sobolev solutions obeying the following time-dependent estimates:
	\begin{align*}
		\|u(t,\cdot)\|_{W^{2,r_n}}&\lesssim\langle t\rangle^{-\frac{n}{2}\left(\frac{1}{m}-\frac{1}{r_n}\right)}\big(\|(u_0,u_1)\|_{\ml{A}_n}+\|(v_0,v_1)\|_{\ml{B}_n}\big),\\
        \|u(t,\cdot)\|_{H^2}&\lesssim\|(u_0,u_1)\|_{\ml{A}_n}+\|(v_0,v_1)\|_{\ml{B}_n},\\
		\|u_t(t,\cdot)\|_{H^1}&\lesssim\langle t\rangle^{-1}\big(\|(u_0,u_1)\|_{\ml{A}_n}+\|(v_0,v_1)\|_{\ml{B}_n}\big),
	\end{align*}
	and
	\begin{align*}
		\|v(t,\cdot)\|_{L^\infty}
		&\lesssim\langle t\rangle^{-\frac{n-1}{2}}
		\big(\|(u_0,u_1)\|_{\ml{A}_n}+\|(v_0,v_1)\|_{\ml{B}_n}\big),\\
		\|v(t,\cdot)\|_{H^2}
		&\lesssim\Lambda_n(t)
		\big(\|(u_0,u_1)\|_{\ml{A}_n}+\|(v_0,v_1)\|_{\ml{B}_n}\big),\\
		\|v_t(t,\cdot)\|_{H^1}
		&\lesssim\|(u_0,u_1)\|_{\ml{A}_n}+\|(v_0,v_1)\|_{\ml{B}_n},
	\end{align*}
	for all $t\geqslant0$.
\end{theorem}

\begin{remark}[Preservation of the time rates]
	The estimates in Theorem~\ref{Thm-01} have the same time rates as the corresponding estimates for the linearized Cauchy problems in Lemma~\ref{Lemma-damped-wave} and Lemma~\ref{Lemma-Wave}.
\end{remark}

\begin{remark}[Intrinsic logarithmic growth in two space dimensions]\label{Remark-Optimal-Logarithmic-Growth}
	We emphasize that the factor $\Lambda_2(t)=\sqrt{\ln(\mathrm{e}+t)}$ is intrinsic to the two-dimensional free wave propagation and is not a loss caused by the nonlinear argument. Indeed, consider the homogeneous free wave equation in $\mb{R}^2$ with initial data $(\varphi_0,\varphi_1)=(0,\varphi_1)$, and set the zeroth moment
	\begin{align*}
		P_{\varphi_1}:=\int_{\mb{R}^2}\varphi_1(x)\,\mathrm{d}x\  \ \mbox{for}\ \ \varphi_1\in L^1.
	\end{align*}
	For suitable $\varphi_1\in H^1\cap L^1$ satisfying $P_{\varphi_1}\neq0$, the next sharp lower bound estimate for the solution $W_1(t,|\nabla|)\varphi_1(x)$ to the free wave equation:
	\begin{align*}
		\|W_1(t,|\nabla|)\varphi_1(\cdot)\|_{L^2}\gtrsim\sqrt{\ln t}\,|P_{\varphi_1}|
	\end{align*}
	holds as $t\gg1$ (see \cite{Ikehata=2023,Chen-Takeda=2023,Chen-Ikehata=2026,Takeda=2026} for its large time optimality). Since the $H^2$ norm dominates the $L^2$ norm, the factor $\Lambda_2(t)$ in \eqref{Wave-Energy} cannot, in general, be replaced by a uniformly bounded function. So, the logarithmic growth
	preserved in Theorem~\ref{Thm-01} reflects an essential
	low-frequency feature of the undamped component.
\end{remark}

\begin{remark}[Comparison with the weighted pointwise framework]
	The recent three-dimensional result in \cite{Georgiev-Kita=2026} covers a larger region of nonlinear exponents by working in weighted space-time $L^\infty$ spaces. At the level of exponent pairs, one has $\Omega_{\mathrm{Nakao},3} \subsetneq \Omega_{\mathrm{GK},3}$. The two theorems nevertheless concern different classes of initial data and solutions. The weighted pointwise framework imposes explicit polynomial spatial decay and provides refined pointwise estimates, whereas Theorem~\ref{Thm-01} gives a mild Sobolev solution in an $H^2\times H^1$-based class, preserves the time rates of the linearized problems, and yields the scattering statement in Corollary~\ref{Cor-Free-Wave}. Thus, the three-dimensional part of Theorem~\ref{Thm-01} should be viewed as complementary to the result in \cite{Georgiev-Kita=2026}.
\end{remark}

As a further consequence of the time-integrability of the source term $|u|^q$ in the undamped wave equation, we obtain the following asymptotic free wave behavior.

\begin{coro}[Asymptotic free wave behavior]\label{Cor-Free-Wave}
	Under the assumptions of Theorem~\ref{Thm-01}, the undamped component $v$ scatters to a free wave in $H^2\times H^1$. More precisely, there exist asymptotic data $(v_0^+,v_1^+)\in H^2\times H^1$ such that
	\begin{align*}
		\lim_{t\to\infty}\sum_{\ell\in\{0,1\}}\left\|\partial_t^\ell v(t,\cdot)-\partial_t^\ell\left(\cos(t|\nabla|)v_0^+(\cdot)+\frac{\sin(t|\nabla|)}{|\nabla|}v_1^+(\cdot)\right)\right\|_{H^{2-\ell}}=0.
	\end{align*}
\end{coro}

\begin{remark}[Scattering only for the undamped component]
	The scattering statement concerns only the undamped component. In three space dimensions, the unweighted time-integrability of $|u|^q$ in $H^1\cap L^1$ is sufficient. In two space dimensions, the low-frequency growth of $W_1(t,|\nabla|)$ requires the slightly stronger weighted integrability with the factor $\sqrt{\ln(\mathrm{e}+t)}$, which is still available because the inequalities in \eqref{Time-integrability-condition} are strict. No analogous free scattering statement is asserted for the damped component.
\end{remark}

\subsection{Comments on the parameter assumptions}

 We collect some comments on the auxiliary exponent $m$, on the dimension-dependent target exponent $r_n$, and on the restrictions appearing in Theorem~\ref{Thm-01}.

\begin{remark}[The auxiliary exponent $m$]
	The intervals in \eqref{Restriction-m-2D} and \eqref{Restriction-m-3D} are non-empty under the assumptions of Theorem~\ref{Thm-01}.
	
	We first consider the two-dimensional case. If $2<q\leqslant4$, then $r_2=q$, and \eqref{Restriction-m-2D} reduces to
	\begin{align*}
		\max\left\{1,\frac{2}{p-2}\right\}<m<\frac{q}{2}.
	\end{align*}
	Thanks to $q>2$, one has $1<\frac{q}{2}$. Moreover, $\frac{2}{p-2}<\frac{q}{2}\Leftrightarrow p>2+\frac{4}{q}$. Hence, the interval is non-empty by the assumption $(p,q)\in\Omega_{\mathrm{Nakao},2}$. If $q>4$, then $r_2=4$ and $\frac{qr_2}{q+r_2}=\frac{4q}{q+4}>2$. Thus, \eqref{Restriction-m-2D} becomes
	\begin{align*}
		\max\left\{1,\frac{2}{p-2}\right\}<m<2,
	\end{align*}
	which is non-empty because $p>3$.
	
	We next consider the three-dimensional case. If $2<q\leqslant\frac{10}{3}$, then $r_3=q$, and \eqref{Restriction-m-3D} reduces to
	\begin{align*}
		\max\left\{1,\frac{2}{p-1},\frac{3q}{2q+3}\right\}<m<\frac{3q}{5}.
	\end{align*}
	Here, $1<\frac{3q}{5}$ and $\frac{3q}{2q+3}<\frac{3q}{5}$ are guaranteed by $q>2$, whereas $\frac{2}{p-1}<\frac{3q}{5}\Leftrightarrow p>1+\frac{10}{3q}$. Therefore, this interval is non-empty. If $q>\frac{10}{3}$, then $r_3=\frac{10}{3}$, and
	\eqref{Restriction-m-3D} becomes
	\begin{align*}
		\max\left\{1,\frac{2}{p-1},\frac{30}{29}\right\}<m<2,
	\end{align*}
	which is non-empty because $p>2$.
	
	Moreover, every admissible $m$ satisfies the following crucial inequalities for $n\in\{2,3\}$:
	\begin{align}\label{Time-integrability-condition}
		\frac{n}{2}\left(\frac{1}{m}-\frac{1}{r_n}\right)<1,\ \ \frac{n-1}{2}\left(p-\frac{2}{m}\right)>1, \ \
		\frac{n-1}{2}(p-1)>1,\ \ q\frac{n}{2}\left(\frac{1}{m}-\frac{1}{r_n}\right)>1.
	\end{align}
	Indeed, the parameter restrictions give the following implications.
	\begin{itemize}
		\item If $n=2$, the first inequality follows from $m>1$, the second from $m>\frac{2}{p-2}$, the third from $p>3$, and the fourth from $m<\frac{qr_2}{q+r_2}$.
		\item If $n=3$, the first inequality follows from $m>\frac{3r_3}{2r_3+3}$, the second from $m>\frac{2}{p-1}$, the third from $p>2$, and the fourth from $m<\frac{3qr_3}{3q+2r_3}$.
	\end{itemize}
	The four inequalities in \eqref{Time-integrability-condition} are used, respectively, to ensure that the polynomial kernel in the $L^m-L^{r_n}$ estimate for the damped wave component has decay order strictly smaller than one, that the $W^{2,m}$ norm of $|v|^p$ is time-integrable, that the $H^1$ and high-frequency norms of $|v|^p$ are time-integrable, and that the source term $|u|^q$ in the undamped wave equation is time-integrable. Again, due to the fact that the three decay exponents subject to time-integrability are strictly larger than $1$, the logarithmic factors appearing in the two-dimensional estimates do not affect the corresponding integrability properties.
\end{remark}

\begin{remark}[The dimension-dependent target exponent]
	The role of $r_n$ is different in the two dimensions.
	\begin{itemize}
		\item If $n=2$, we choose $r_2=\min\{q,4\}$. This allows $q$ to be arbitrarily large and guarantees the Sobolev-Bessel embedding $H^2\hookrightarrow W^{2-\frac{2}{r_2},r_2}$ required for the high-frequency part of the damped wave estimate.
		\item If $n=3$, we choose $r_3=\min\{q,\frac{10}{3}\}$. The bound $r_3\leqslant\frac{10}{3}$ guarantees the Sobolev-Bessel embedding $H^2\hookrightarrow W^{2-\frac{2}{r_3},r_3}$. When $q>\frac{10}{3}$, one has $r_3=\frac{10}{3}>3$ and hence the Sobolev-Morrey embedding $W^{2,r_3}\hookrightarrow W^{1,\infty}$ holds, which provides the $W^{1,q}$ and $W^{1,2q}$ bounds required for $|u|^q$. Thus, within the present $H^2$-based solution framework, the nonlinear estimates impose no additional upper restriction on $q$.
	\end{itemize}
\end{remark}

The shift between the two-dimensional and three-dimensional admissible ranges is illustrated in Figure~\ref{Fig-3D-2D-comparison}. In particular, the two-dimensional range is shifted toward larger powers because the Poisson formula yields only the decay rate $\langle t\rangle^{-\frac12}$ for the undamped component, whereas the three-dimensional Kirchhoff formula gives the stronger decay rate $\langle t\rangle^{-1}$. Because of $1+\frac{10}{3q}<\frac{8}{3}<3$ for all $q>2$, we claim the strict inclusion 
\begin{align*}
\Omega_{\mathrm{Nakao},2} \subsetneq \Omega_{\mathrm{Nakao},3}.
\end{align*}

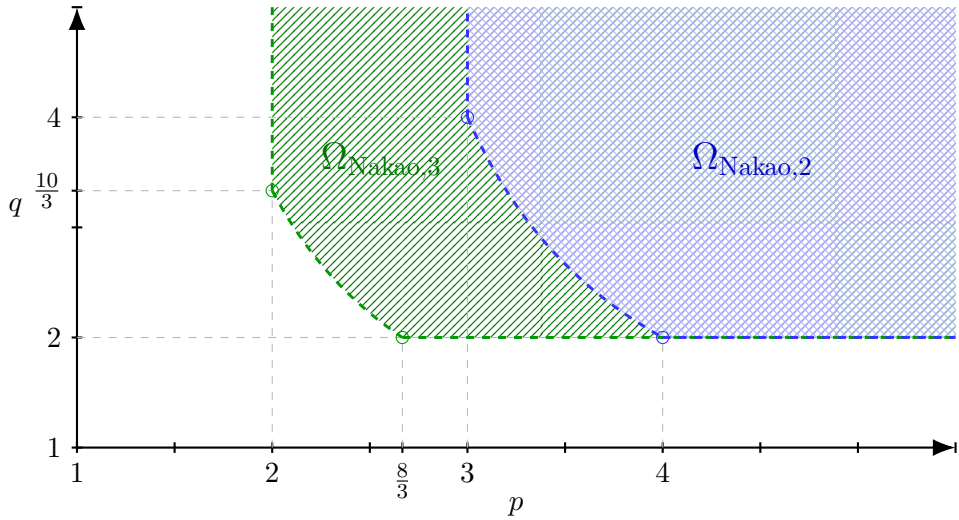
\begin{figure}[!htbp]
	\centering
	\begin{tikzpicture}
		\begin{axis}[
			width=13.2cm,
			height=7.4cm,
			xmin=1, xmax=5.5,
			ymin=1, ymax=5,
			axis lines=left,
			axis line style={thick, -{Latex[length=3mm]}},
			xlabel={$p$},
			ylabel={$q$},
			xlabel style={at={(axis description cs:0.5,-0.09)}},
			ylabel style={at={(axis description cs:-0.07,0.5)},rotate=270,anchor=south},
			xtick={1,1.5,2,2.5,8/3,3,3.5,4,4.5,5,5.5},
			xticklabels={$1$,$ $,$2$,$ $,$\frac83$,$3$,$ $,$4$,$ $,$ $,$ $},
			ytick={1,2,3,10/3,4,5},
			yticklabels={$1$,$2$,$ $,$\frac{10}{3}$,$4$,$ $},
			tick style={black,thick},
			major tick length=4pt,
			clip=false,
			]
			
			\addplot[gray!55,dashed,thin,domain=1:5.5] {2};
			\addplot[gray!55,dashed,thin,domain=1:2] {10/3};
			\addplot[gray!55,dashed,thin,domain=1:3] {4};
			\addplot[gray!55,dashed,thin] coordinates {(2,1) (2,10/3)};
			\addplot[gray!55,dashed,thin] coordinates {(8/3,1) (8/3,2)};
			\addplot[gray!55,dashed,thin] coordinates {(3,1) (3,4)};
			\addplot[gray!55,dashed,thin] coordinates {(4,1) (4,2)};
			
			
			\addplot[name path=green top left,draw=none,domain=2:8/3] {5};
			\addplot[name path=green curve,draw=none,domain=2:8/3,samples=200] {10/(3*(x-1))};
			\addplot[name path=green top right,draw=none,domain=8/3:5.5] {5};
			\addplot[name path=green bottom,draw=none,domain=8/3:5.5] {2};
			\addplot[draw=none,fill=green!8,pattern=north east lines,pattern color=green!45!black]fill between[of=green top left and green curve];
			\addplot[draw=none,fill=green!8,pattern=north east lines,pattern color=green!45!black]fill between[of=green top right and green bottom];
			
			
			\addplot[name path=blue top left,draw=none,domain=3:4] {5};
			\addplot[name path=blue curve,draw=none,domain=3:4,samples=200] {4/(x-2)};
			\addplot[name path=blue top right,draw=none,domain=4:5.5] {5};
			\addplot[name path=blue bottom,draw=none,domain=4:5.5] {2};
			\addplot[draw=none,fill=blue!6,pattern=crosshatch,pattern color=blue!30]fill between[of=blue top left and blue curve];
			\addplot[draw=none,fill=blue!6,pattern=crosshatch,pattern color=blue!30]fill between[of=blue top right and blue bottom];
			
			
			\addplot[green!60!black,dashed,line width=1.1pt] coordinates {(2,10/3) (2,5)};

			\addplot[green!60!black,dashed,line width=1.1pt,domain=2:8/3,samples=200] {10/(3*(x-1))};
			
			\addplot[green!60!black,dashed,line width=1.1pt,domain=8/3:5.5] {2};
			
			
			\addplot[blue!80,dashed,line width=1.1pt] coordinates {(3,4) (3,5)};
			
			\addplot[blue!80,dashed,line width=1.1pt,domain=3:4,samples=200] {4/(x-2)};
			
			\addplot[blue!80,dashed,line width=1.1pt,domain=4:5.5] {2};
			
			\addplot[only marks,mark=o,mark size=2.3pt,color=green!60!black] coordinates {(2,10/3)(8/3,2)};
			\addplot[only marks,mark=o,mark size=2.3pt,color=blue!80] coordinates {(3,4)(4,2)};
			
			\node[green!45!black,font=\Large] at (axis cs:2.56,3.6){$\Omega_{\mathrm{Nakao},3}$};
			\node[blue!75!black,font=\Large] at (axis cs:4.46,3.6){$\Omega_{\mathrm{Nakao},2}$};
			
		\end{axis}
	\end{tikzpicture}
	\caption{Comparison of the admissible regions in 2D and 3D for global in-time existence.}
	\label{Fig-3D-2D-comparison}
\end{figure}

\subsection{Comparison with known blow-up ranges}
Figure~\ref{Fig-Range-2D} and Figure~\ref{Fig-Range-3D} compare the global in-time existence regions obtained in Theorem~\ref{Thm-01} with the known blow-up ranges from \cite{Wakasugi=2017,Chen-Reissig=2021,Kita-Kusaba=2022}, together with the critical curves for the wave-wave and damped-damped wave benchmark systems, in the $p-q$ plane. 

Figure~\ref{Fig-Range-3D} displays the exponent range obtained by the
present Sobolev argument. It is not intended to represent the full
currently known global existence range in three space dimensions.
The larger weighted pointwise region established in
\cite{Georgiev-Kita=2026} is described separately above.

Actually, these two figures reveal a common effect of the friction term $+u_t$ in Nakao's problem \eqref{Nakao-Problem}. In both dimensions $n\in\{2,3\}$, the global in-time existence region obtained in Theorem~\ref{Thm-01} is not confined to the expected global in-time existence range of the purely wave-wave benchmark system \eqref{Wave-System}. A non-empty portion of $\Omega_{\mathrm{Nakao},n}$ lies beyond the Strauss-type threshold $\Gamma_{\mathrm{Wave},n}(p,q)=0$. Thus, although the $v$-component remains undamped and retains its wave-like propagation, dissipation in the $u$-component already weakens the nonlinear feedback between the two equations and enlarges the range in which the coupled system can be controlled. This comparison does not identify the sharp critical curve for Nakao's problem, but it demonstrates that damping in only one component has a genuine stabilizing effect at least in both space dimensions considered here.

The region between the presently known blow-up and global in-time existence ranges remains open and is expected to contain the sharp threshold separating these two regimes.

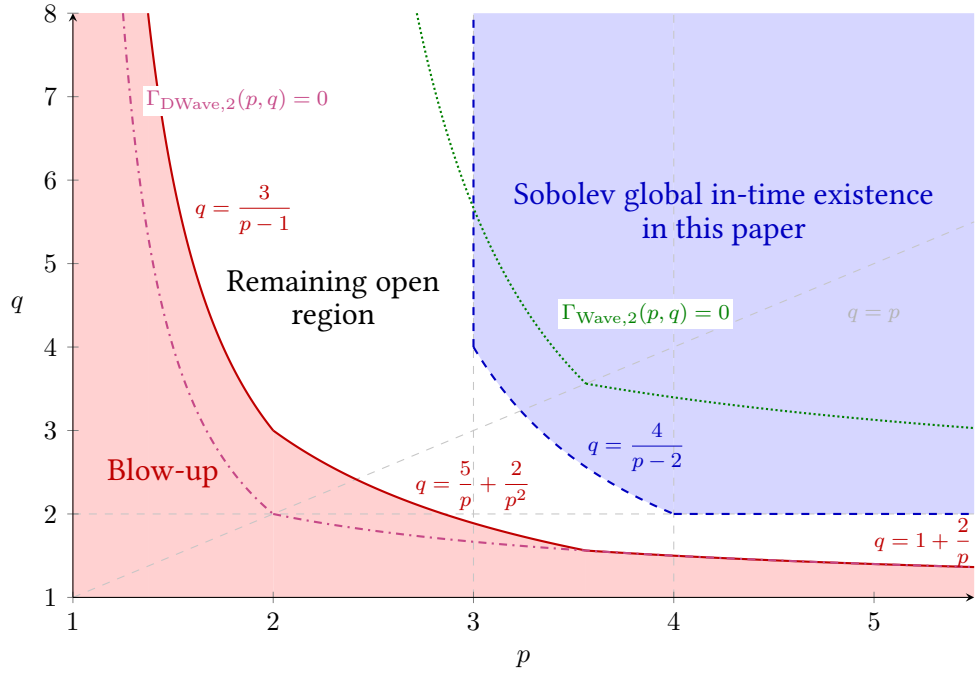
\begin{figure}[!htbp]
	\centering
	\begin{tikzpicture}
		\pgfmathsetmacro{\qmax}{8}
		\pgfmathsetmacro{\ptop}{11/8}
		\pgfmathsetmacro{\pstar}{(3+sqrt(17))/2}
		\pgfmathsetmacro{\pglo}{4}
		\pgfmathsetmacro{\pdwleft}{5/4}
		\pgfmathsetmacro{\pwleft}{(21+sqrt(505))/16}
		
		\begin{axis}[
			width=0.90\textwidth,
			height=0.62\textwidth,
			xmin=1, xmax=5.5,
			ymin=1, ymax=\qmax,
			axis lines=left,
			xlabel={$p$},
			ylabel={$q$},
			ylabel style={rotate=270},
			xtick={1,2,3,4,5},
			ytick={1,2,3,4,5,6,7,8},
			clip=true,
			axis on top=true,
			tick label style={font=\small},
			label style={font=\small},
			]
			
			\addplot[name path=gupper, draw=none, domain=3:5.5, samples=2]{\qmax};
			\addplot[name path=glowera, draw=none, domain=3:\pglo, samples=200]{4/(x-2)};
			\addplot[name path=glowerb, draw=none, domain=\pglo:5.5, samples=2]{2};
			
			\addplot[name path=bottomone, draw=none, domain=1:5.5, samples=2]{1};
			\addplot[name path=blowcurveA, draw=none, domain=\ptop:2, samples=200]{3/(x-1)};
			\addplot[name path=blowcurveB, draw=none, domain=2:\pstar, samples=200]{5/x+2/(x*x)};
			\addplot[name path=blowcurveC, draw=none, domain=\pstar:5.5, samples=200]{1+2/x};
			
			\path[fill=red!18](axis cs:1,1) -- (axis cs:\ptop,1) -- (axis cs:\ptop,\qmax) -- (axis cs:1,\qmax) -- cycle;
			\addplot[draw=none,fill=red!18,]fill between[of=blowcurveA and bottomone,soft clip={domain=\ptop:2}];
			\addplot[draw=none,fill=red!18,]fill between[of=blowcurveB and bottomone,soft clip={domain=2:\pstar}];
			\addplot[draw=none,fill=red!18,]fill between[of=blowcurveC and bottomone,soft clip={domain=\pstar:5.5}];
			
			\addplot[draw=none,fill=blue!16,]fill between[of=gupper and glowera,soft clip={domain=3:\pglo}];
			\addplot[draw=none,fill=blue!16,]fill between[of=gupper and glowerb,soft clip={domain=\pglo:5.5}];
			
			\addplot[gray!45, thin, dashed, domain=1:5.5, samples=2]{x};
			\addplot[gray!45, thin, dashed, domain=1:5.5, samples=2]{2};
			\addplot[gray!45, thin, dashed]coordinates {(3,1) (3,\qmax)};
			\addplot[gray!45, thin, dashed]coordinates {(4,1) (4,\qmax)};
			
			\addplot[red!75!black, thick, domain=\ptop:2, samples=200]{3/(x-1)};
			\addplot[red!75!black, thick, domain=2:\pstar, samples=200]{5/x+2/(x*x)};
			\addplot[red!75!black, thick, domain=\pstar:5.5, samples=200]{1+2/x};
			
			\addplot[blue!75!black, thick, dashed, domain=3:\pglo, samples=200]{4/(x-2)};
			\addplot[blue!75!black, thick, dashed, domain=\pglo:5.5, samples=2]{2};
			\addplot[blue!75!black, thick, dashed] coordinates {(3,4) (3,\qmax)};
			\addplot[magenta!80!black,thick,dash dot,domain=\pdwleft:2,samples=200,]{2/(x-1)};
			\addplot[magenta!80!black,thick,dash dot,domain=2:5.5,samples=200,]{1+2/x};
			
			\addplot[green!50!black,thick,densely dotted,domain=\pwleft:\pstar,samples=200,]{(5+2/x)/(x-2)};
			\addplot[green!50!black,thick,densely dotted,domain=\pstar:5.5,samples=200,]{(2*x+5+sqrt((2*x+5)^2+8*x))/(2*x)};
			
			\node[font=\small, red!75!black, align=center] at (axis cs:1.45,2.50){\large Blow-up};
			\node[font=\small, blue!75!black, align=center] at (axis cs:4.25,5.60){\large Sobolev global in-time existence\\ \large in this paper};
			\node[font=\small, fill=white, inner sep=1.2pt, align=center] at (axis cs:2.3,4.55){\large Remaining open\\ \large region};
			
			\node[font=\scriptsize, red!75!black,anchor=west] at (axis cs:1.56,5.65){$q=\dfrac{3}{p-1}$};
			\node[font=\scriptsize, red!75!black,anchor=west] at (axis cs:2.66,2.34){$q=\dfrac{5}{p}+\dfrac{2}{p^2}$};
			\node[font=\scriptsize, red!75!black,anchor=west] at (axis cs:4.95,1.66){$q=1+\dfrac{2}{p}$};
			\node[font=\scriptsize, blue!75!black,anchor=west] at (axis cs:3.52,2.8){$q=\dfrac{4}{p-2}$};
			
			\node[font=\scriptsize\bfseries,text=magenta!80!black,fill=white,inner sep=1.2pt] at (axis cs:1.82,6.95){$\Gamma_{\mathrm{DWave},2}(p,q)=0$};
			\node[font=\scriptsize\bfseries,text=green!50!black,fill=white,inner sep=1.2pt] at (axis cs:3.86,4.39){$\Gamma_{\mathrm{Wave},2}(p,q)=0$};
			
			\node[font=\scriptsize, gray!60, anchor=south east] at (axis cs:5.18,4.18){$q=p$};
			
		\end{axis}
	\end{tikzpicture}
	\caption{Blow-up versus global in-time existence regions for Nakao's problem in 2D.}
	\label{Fig-Range-2D}
\end{figure}

\bigskip

\begin{figure}[!htbp]
	\centering
	\begin{tikzpicture}
		\pgfmathsetmacro{\pzero}{(3+sqrt(21))/6}
		\pgfmathsetmacro{\pstar}{(2+sqrt(10))/2}
		\pgfmathsetmacro{\pglo}{8/3}
		\pgfmathsetmacro{\ps}{1+sqrt(2)}
		\pgfmathsetmacro{\pdwleft}{13/12}
		\pgfmathsetmacro{\pwleft}{(7+sqrt(65))/8}
		
		\begin{axis}[
			width=0.90\textwidth,
			height=0.62\textwidth,
			xmin=1, xmax=4,
			ymin=1, ymax=3.75,
			axis lines=left,
			xlabel={$p$},
			ylabel={$q$},
			ylabel style={rotate=270},
			xtick={1,2,3},
			ytick={1,2,3},
			extra y ticks={10/3},
			extra y tick labels={$\frac{10}{3}$},
			clip=true,
			axis on top=true,
			tick label style={font=\small},
			label style={font=\small},
			]
			
			\addplot[name path=gupper,draw=none,domain=2:4,samples=2]{3.75};
			\addplot[name path=glowera,draw=none,domain=2:\pglo,samples=200]{10/(3*(x-1))};
			\addplot[name path=glowerb,draw=none,domain=\pglo:4,samples=2]{2};
			
			\addplot[name path=bottomone,draw=none,domain=1:4,samples=2]{1};
			\addplot[name path=blowcurveA,draw=none,domain=\pzero:\pstar,samples=200]{(3*x+1)/(x*x)};
			\addplot[name path=blowcurveB,draw=none,domain=\pstar:3,samples=200]{(2*x+5)/(3*x)};
			
			\path[fill=red!18](axis cs:1,1)--(axis cs:\pzero,1)--(axis cs:\pzero,3)--(axis cs:1,3)--cycle;
			\addplot[draw=none,fill=red!18]fill between[of=blowcurveA and bottomone,soft clip={domain=\pzero:\pstar}];
		    \addplot[draw=none,fill=red!18]fill between[of=blowcurveB and bottomone,soft clip={domain=\pstar:3}];
			
			\addplot[draw=none,fill=blue!16]fill between[of=gupper and glowera,soft clip={domain=2:\pglo}];
			\addplot[draw=none,fill=blue!16]fill between[of=gupper and glowerb,soft clip={domain=\pglo:4}];
			
			\addplot[gray!45,thin,dashed,domain=1:3.75,samples=2]{x};
			\addplot[gray!45,thin,dashed,domain=1:4,samples=2]{2};
			\addplot[gray!45,thin,dashed,domain=1:4,samples=2]{10/3};
			\addplot[gray!45,thin,dashed,domain=1:4,samples=2]{3};
			\addplot[gray!45,thin,dashed]coordinates {(2,1) (2,3.75)};
			\addplot[gray!45,thin,dashed]coordinates {(3,1) (3,3.75)};
			
			\addplot[red!75!black,thick,domain=1:\pzero,samples=2]{3};
			\addplot[red!75!black,thick,domain=\pzero:\pstar,samples=200]{(3*x+1)/(x*x)};
			\addplot[red!75!black,thick,domain=\pstar:3,samples=200]{(2*x+5)/(3*x)};
			
			\addplot[blue!75!black,thick,dashed]coordinates {(2,10/3) (2,3.75)};
			\addplot[blue!75!black,thick,dashed,domain=2:\pglo,samples=200]{10/(3*(x-1))};
			\addplot[blue!75!black,thick,dashed,domain=\pglo:4,samples=2]{2};
			
			\addplot[magenta!80!black,thick,dash dot,domain=\pdwleft:5/3,samples=200]{5/(3*x-2)};
			\addplot[magenta!80!black,thick,dash dot,domain=5/3:4,samples=200]{(2*x+5)/(3*x)};
			
			\addplot[green!50!black,thick,densely dotted,domain=\pwleft:\ps,samples=200]{(3*x+1)/(x*(x-1))};
			\addplot[green!50!black,thick,densely dotted,domain=\ps:4,samples=200]{(x+3+sqrt(x^2+10*x+9))/(2*x)};
			
			\node[font=\small,red!75!black]at (axis cs:1.30,1.65){\large Blow-up};
			\node[font=\small,blue!75!black,align=center]at (axis cs:3.15,2.44){\large Sobolev global in-time existence\\ \large in this paper};
			\node[font=\small,fill=white,inner sep=1.2pt]at (axis cs:3,1.73){\large Not covered by the Sobolev argument};
			
			\node[font=\scriptsize,red!75!black,anchor=west]at (axis cs:1.56,2.38){$q=\dfrac{3}{p}+\dfrac{1}{p^2}$};
			\node[font=\scriptsize,red!75!black,anchor=west]at (axis cs:2.57,1.4){$q=\dfrac{2p+5}{3p}$};
			\node[font=\scriptsize,blue!75!black,anchor=west]at (axis cs:2.22,2.88){$q=\dfrac{10}{3(p-1)}$};
			
			\node[font=\scriptsize\bfseries,text=magenta!80!black,fill=white,inner sep=1.2pt]at (axis cs:1.56,3.10){$\Gamma_{\mathrm{DWave},3}(p,q)=0$};
			\node[font=\scriptsize\bfseries,text=green!50!black,fill=white,inner sep=1.2pt]at (axis cs:2.36,3.52){$\Gamma_{\mathrm{Wave},3}(p,q)=0$};
			
			\node[font=\scriptsize,gray!60,anchor=south east]at (axis cs:3.47,3.47){$q=p$};
			
		\end{axis}
	\end{tikzpicture}
	\caption{Blow-up versus global in-time existence regions for Nakao's problem in 3D.}
	\label{Fig-Range-3D}
\end{figure}
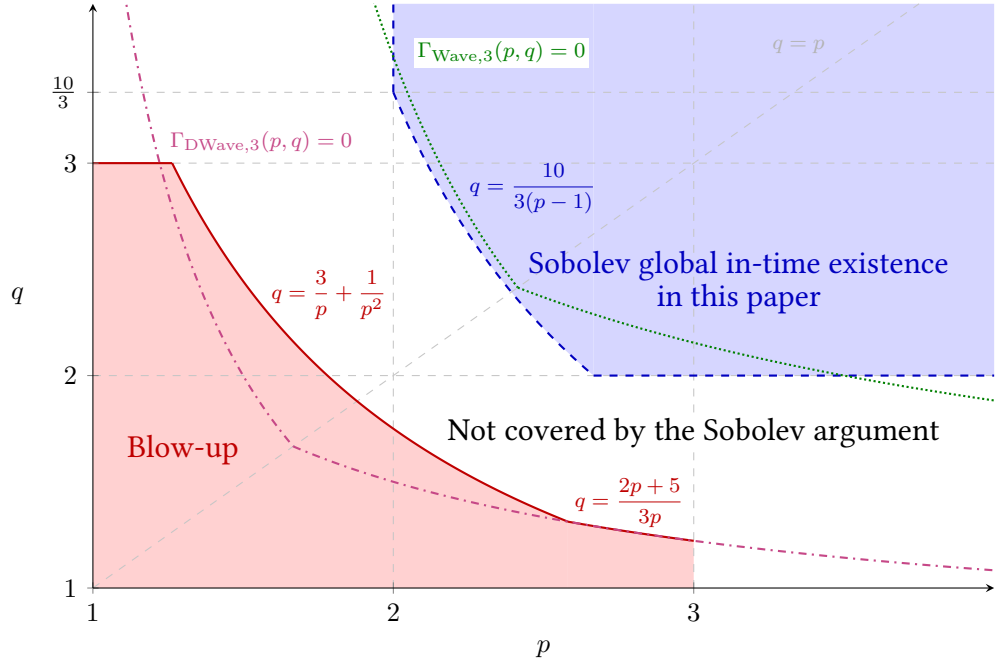

\section{Preliminaries}
 This section provides the estimates for the linearized Cauchy problems required in the nonlinear argument. After recording an elementary time-convolution estimate, we establish diffusion-type $L^m-L^r$ upper bounds for the damped wave equation and derive dimension-dependent estimates for the free wave equation from the Poisson and Kirchhoff formulas. Particular attention is paid to the logarithmic low-frequency growth in two space dimensions and to the different $L^\infty$ decay rates in dimensions two and three.

\begin{lemma}[Time-convolution estimates]\label{Lemma-Time-Convolution}
	Let $0\leqslant a\leqslant1$, $b>1$ and $\gamma\geqslant0$. Then, the following estimates:
	\begin{align*}
		\int_0^t\langle t-\tau\rangle^{-a}\langle\tau\rangle^{-b}[\ln(\mathrm{e}+\tau)]^\gamma\,\mathrm{d}\tau&\lesssim\langle t\rangle^{-a},\\
		\int_0^t\mathrm{e}^{-c(t-\tau)}\langle\tau\rangle^{-b}[\ln(\mathrm{e}+\tau)]^\gamma\,\mathrm{d}\tau&\lesssim\langle t\rangle^{-b}[\ln(\mathrm{e}+t)]^\gamma,
	\end{align*}
	hold for all $t\geqslant0$. 
\end{lemma}

\begin{proof}
	The estimates are immediate for bounded $t$. For $t\geqslant2$, we split each integral at $\frac{t}{2}$. From $\langle t-\tau\rangle\approx\langle t\rangle$ on $[0,\frac{t}{2}]$ and $b>1$, we have
	\begin{align*}
		\int_0^{\frac{t}{2}}\langle t-\tau\rangle^{-a}\langle\tau\rangle^{-b}[\ln(\mathrm{e}+\tau)]^\gamma\,\mathrm{d}\tau\lesssim\langle t\rangle^{-a}.
	\end{align*}
	On $[\frac{t}{2},t]$, we use $\langle\tau\rangle\approx\langle t\rangle$ and integrate the remaining kernel. The resulting bound is
	\begin{align*}
		\langle t\rangle^{1-a-b}[\ln(\mathrm{e}+t)]^\gamma\lesssim\langle t\rangle^{-a}\ \ &\mbox{if}\ \ 0\leqslant a<1,\\
		\langle t\rangle^{-b}[\ln(\mathrm{e}+t)]^{1+\gamma}\lesssim\langle t\rangle^{-a}\ \ &\mbox{if}\ \ a=1,
	\end{align*}
	in which the last estimate follows from $b>1$. This proves the first assertion. The second one follows from the same decomposition, using the exponential decay on $[0,\frac{t}{2}]$ and $\langle\tau\rangle\approx\langle t\rangle$ together with $\int_0^\infty\mathrm{e}^{-cs}\,\mathrm{d}s<\infty$ on $[\frac{t}{2},t]$.
\end{proof}

\subsection{Estimates for the linear damped wave equation}

 We focus on the following linear damped wave equation:
\begin{align}\label{Damped-Wave}
	\begin{cases}
		\phi_{tt}-\Delta\phi+\phi_t=F(t,x),&x\in\mb{R}^n,\ t>0,\\
		(\phi,\phi_t)(0,x)=(\phi_0,\phi_1)(x),&x\in\mb{R}^n,
	\end{cases}
\end{align}
which is the inhomogeneous version of \eqref{Nakao-Problem}$_1$.

\begin{lemma}[Damped wave estimates]\label{Lemma-damped-wave}
	Let $n\in\{2,3\}$, $1<p_0<2<q_0<\infty$ and $s_0\in\{0,1\}$. Set
	\begin{align*}
		\beta_n(q_0):=(n-1)\left(\frac12-\frac1{q_0}\right).
	\end{align*}
	Then, the solution to the inhomogeneous damped wave equation \eqref{Damped-Wave} satisfies
	\begin{align*}
		\|\phi(t,\cdot)\|_{W^{s_0+1,q_0}}&\lesssim\langle t\rangle^{-\frac n2\left(\frac1{p_0}-\frac1{q_0}\right)}
		\|(\phi_0,\phi_1)\|_{W^{s_0+1,p_0}\times W^{s_0+1,p_0}}\\[0.3em]
        &\quad\ +\mathrm{e}^{-ct}\|(\phi_0,\phi_1)\|_{W^{s_0+1+\beta_n(q_0),q_0}\times W^{s_0+\beta_n(q_0),q_0}}\\[0.3em]
		&\quad\ +\int_0^t\langle t-\tau\rangle^{-\frac n2\left(\frac1{p_0}-\frac1{q_0}\right)}\|F(\tau,\cdot)\|_{W^{s_0+1,p_0}}\,\mathrm{d}\tau\\
		&\quad\ +\int_0^t\mathrm{e}^{-c(t-\tau)}\|F(\tau,\cdot)\|_{W^{s_0+\beta_n(q_0),q_0}}\,\mathrm{d}\tau,
	\end{align*}
and
\begin{align*}
	\|\phi(t,\cdot)\|_{H^{s_0+1}}&\lesssim\|(\phi_0,\phi_1)\|_{H^{s_0+1}\times H^{s_0}}+\int_0^t\|F(\tau,\cdot)\|_{H^{s_0}}\,\mathrm{d}\tau,\\
	\|\phi_t(t,\cdot)\|_{H^{s_0}}&\lesssim\langle t\rangle^{-1}\|(\phi_0,\phi_1)\|_{H^{s_0+1}\times H^{s_0}}+\int_0^t\langle t-\tau\rangle^{-1}\|F(\tau,\cdot)\|_{H^{s_0}}\,\mathrm{d}\tau,
\end{align*}
for all $t\geqslant 0$.
\end{lemma}

\begin{proof}
	Let $D(t)$ be the operator used in \cite[Theorem~1.1]{Ikeda-Inui-Okamoto-Wakasugi=2019}. The homogeneous solution is represented via
	\begin{align*}
		\phi^{\mathrm{hom}}(t,x)=D(t)\big(\phi_0(x)+\phi_1(x)\big)+\partial_tD(t)\phi_0(x).
	\end{align*}
	In the notation of that theorem, the output exponent is $q_0$, the input exponent is $p_0$ and
	\begin{align*}
		\beta=(n-1)\left|\frac12-\frac1{q_0}\right|=\beta_n(q_0).
	\end{align*}
	We apply its estimate for $D(t)$ and $\partial_tD(t)$ with $s_1=s_2=s_0+1$. The low-frequency terms give the decay factor 
	\begin{align*}
		\langle t\rangle^{-\frac n2\left(\frac1{p_0}-\frac1{q_0}\right)}.
	\end{align*} For the high-frequency terms, the estimate for $D(t)$ requires $s_0+\beta_n(q_0)$ derivatives of the datum, while the estimate for $\partial_tD(t)$ requires $s_0+1+\beta_n(q_0)$ derivatives. The polynomial factor accompanying the exponential decay in \cite[Theorem~1.1]{Ikeda-Inui-Okamoto-Wakasugi=2019} is absorbed into $\mathrm{e}^{-ct}$. Duhamel's principle suggests
	\begin{align*}
		\phi(t,x)=D(t)\big(\phi_0(x)+\phi_1(x)\big)+\partial_tD(t)\phi_0(x)+\int_0^tD(t-\tau)F(\tau,x)\,\mathrm{d}\tau.
	\end{align*}
	Applying the estimate for $D(t-\tau)$ to $F(\tau,\cdot)$ and using Minkowski's inequality yields the two Duhamel terms in the first estimate. 
	
	The remaining estimates follow from the standard $L^2$-based
	bounds for the damped wave propagators. More precisely, for
	$s_0\in\{0,1\}$, we are able to derive
	\begin{align*}
		\|D(t)f(\cdot)\|_{H^{s_0+1}}&\lesssim \|f\|_{H^{s_0}},\\
		\|\partial_tD(t)f(\cdot)\|_{H^{s_0}}&\lesssim \langle t\rangle^{-1}\|f\|_{H^{s_0}},\\
		\|\partial_tD(t)f(\cdot)\|_{H^{s_0+1}}&\lesssim \|f\|_{H^{s_0+1}},\\
		\|\partial_t^2D(t)f(\cdot)\|_{H^{s_0}}&\lesssim \langle t\rangle^{-1}\|f\|_{H^{s_0+1}}.
	\end{align*}
	Applying these estimates to the homogeneous representation and the
	Duhamel term, and then using Minkowski's inequality, gives the second
	and third assertions.
\end{proof}

\subsection{Linear wave estimates via the Poisson and Kirchhoff formulas}

 We next consider the following linear wave equation:
\begin{align}\label{Wave}
	\begin{cases}
		\varphi_{tt}-\Delta\varphi=G(t,x),&x\in\mb{R}^n,\ t>0,\\
		(\varphi,\varphi_t)(0,x)=(\varphi_0,\varphi_1)(x),&x\in\mb{R}^n,
	\end{cases}
\end{align}
which is the inhomogeneous version of \eqref{Nakao-Problem}$_2$.

\begin{lemma}[Spherical trace estimate]\label{Lemma-Trace}
	Let $n\in\{2,3\}$ and $h\in W^{1,1}$. Then, there exists a constant $C>0$, independent of $x\in\mb{R}^n$ and $R\geqslant1$, such that
	\begin{align*}
		\int_{|x-y|=R}|h(y)|\,\mathrm{d}S_y\leqslant C\|h\|_{W^{1,1}}.
	\end{align*}
\end{lemma}

\begin{proof}
	Let us first assume that $h\in\ml{C}_0^1$ and $x=0$. Writing $y=R\omega$ with $\omega\in\mb{S}^{n-1}$, one has $\mathrm{d}S_y=R^{n-1}\,\mathrm{d}\omega$. For every $r\in[R,R+1]$, the fundamental theorem of calculus gives
	\begin{align*}
		|h(R\omega)|&\leqslant |h(r\omega)|+\int_R^r|\omega\cdot\nabla h(\rho\,\omega)|\,\mathrm{d}\rho\\
		&\leqslant |h(r\omega)|+\int_R^{R+1}|\nabla h(\rho\,\omega)|\,\mathrm{d}\rho.
	\end{align*}
	Integrating this inequality with respect to $r\in[R,R+1]$, multiplying it by $R^{n-1}$ and using $R^{n-1}\leqslant r^{n-1}$ for $r\in[R,R+1]$, we obtain
	\begin{align}\label{Trace-at-origin}
		\int_{|y|=R}|h(y)|\,\mathrm{d}S_y&=R^{n-1}\int_{\mb{S}^{n-1}}|h(R\omega)|\,\mathrm{d}\omega\notag\\
		&\lesssim\int_{\mb{S}^{n-1}}\int_R^{R+1}\big(|h(r\omega)|+|\nabla h(r\omega)|\big)r^{n-1}\,\mathrm{d}r\,\mathrm{d}\omega\notag\\[0.3em]
		&\lesssim\|h\|_{L^1}+\|\nabla h\|_{L^1}.
	\end{align}
	For a general center $x\in\mb{R}^n$, let $h_x(z):=h(x+z)$. Applying \eqref{Trace-at-origin} to $h_x$ gives
	\begin{align*}
		\int_{|x-y|=R}|h(y)|\,\mathrm{d}S_y&=\int_{|z|=R}|h_x(z)|\,\mathrm{d}S_z\\[0.3em]
		&\leqslant C\|h_x\|_{W^{1,1}}=C\|h\|_{W^{1,1}}.
	\end{align*}
	Finally, the assertion for $h\in W^{1,1}$ follows from the density of $\ml{C}_0^\infty$ in $W^{1,1}$. The proof now is complete.
\end{proof}

\begin{lemma}[Wave estimates]\label{Lemma-Wave}
	Let $n\in\{2,3\}$ and $s_1\in\{0,1\}$. Then, the solution to the inhomogeneous linear wave equation \eqref{Wave} satisfies
	\begin{align}\label{Wave-Linfty}
		\|\varphi(t,\cdot)\|_{L^\infty}&\lesssim\langle t\rangle^{-\frac{n-1}{2}}\|(\varphi_0,\varphi_1)\|_{(W^{2,1}\cap W^{1,\infty})\times(W^{1,1}\cap L^\infty)}\notag\\[0.3em]
		&\quad\ +\int_0^t\langle t-\tau\rangle^{-\frac{n-1}{2}}\|G(\tau,\cdot)\|_{W^{1,1}\cap L^\infty}\,\mathrm{d}\tau,
		\end{align}
		and
		\begin{align}
		\|\varphi(t,\cdot)\|_{H^{s_1+1}}&\lesssim\Lambda_n(t)\|(\varphi_0,\varphi_1)\|_{H^{s_1+1}\times(H^{s_1}\cap L^1)}+\int_0^t\Lambda_n(t-\tau)\|G(\tau,\cdot)\|_{H^{s_1}\cap L^1}\,\mathrm{d}\tau,	\label{Wave-Energy}\\
		\label{Wave-Time-Derivative}
		\|\varphi_t(t,\cdot)\|_{H^{s_1}}&\lesssim\|(\varphi_0,\varphi_1)\|_{H^{s_1+1}\times H^{s_1}}+\int_0^t\|G(\tau,\cdot)\|_{H^{s_1}}\,\mathrm{d}\tau,
	\end{align}
	for all $t\geqslant0$.
\end{lemma}

\begin{proof} 
	Let us denote the standard free wave propagators via
	\begin{align*}
		W_0(t,|\nabla|):=\cos(t|\nabla|)\ \ \mbox{and}\ \ W_1(t,|\nabla|):=\frac{\sin(t|\nabla|)}{|\nabla|}.
	\end{align*}
	The mild solution to \eqref{Wave} is represented by
	\begin{align}\label{Wave-mild-formula}
		\varphi(t,x)=W_0(t,|\nabla|)\varphi_0(x)+W_1(t,|\nabla|)\varphi_1(x)+\int_0^tW_1(t-\tau,|\nabla|)G(\tau,x)\,\mathrm{d}\tau.
	\end{align}

	We first prove \eqref{Wave-Energy} and \eqref{Wave-Time-Derivative}. The estimates for the terms
	\begin{align*}
		W_0(t,|\nabla|)\varphi_0(\cdot),\ \ \partial_tW_0(t,|\nabla|)\varphi_0(\cdot),\ \ \partial_tW_1(t,|\nabla|)\varphi_1(\cdot)
	\end{align*}
	follow directly from the boundedness of the corresponding Fourier multipliers on Sobolev spaces. The only term that may exhibit low-frequency growth is
	$W_1(t,|\nabla|)\varphi_1(\cdot)$ in $H^{s_1+1}$. On the region $|\xi|\geqslant1$, one has
	\begin{align*}
		(1+|\xi|^2)^{s_1+1}\frac{|\sin(t|\xi|)|^2}{|\xi|^2}\lesssim(1+|\xi|^2)^{s_1},
	\end{align*}
	and hence the high-frequency part is bounded by $\|\varphi_1\|_{H^{s_1}}$. For the low-frequency part, the Hausdorff-Young inequality gives
	\begin{align}\label{Wave-low-frequency-integral}
		\int_{|\xi|\leqslant1}(1+|\xi|^2)^{s_1+1}\frac{|\sin(t|\xi|)|^2}{|\xi|^2}|\hat{\varphi}_1(\xi)|^2\,\mathrm{d}\xi\lesssim\|\varphi_1\|_{L^1}^2\int_0^1|\sin(t\rho)|^2\rho^{n-3}\,\mathrm{d}\rho.
	\end{align}
	If $n=3$, the last integral is uniformly bounded. If $n=2$ and $t\geqslant1$, we split the previous integral at $\rho=t^{-1}$ to obtain
	\begin{align*}
		\int_0^1\frac{|\sin(t\rho)|^2}{\rho}\,\mathrm{d}\rho&\lesssim t^2\int_0^{t^{-1}}\rho\,\mathrm{d}\rho+\int_{t^{-1}}^1\frac{1}{\rho}\,\mathrm{d}\rho\\[0.3em]
		&\lesssim\ln(\mathrm{e}+t).
	\end{align*}
	The same bound is immediate for $0\leqslant t\leqslant1$. As a consequence,
	\begin{align}\label{W1-energy-bound}
		\|W_1(t,|\nabla|)\varphi_1(\cdot)\|_{H^{s_1+1}}\lesssim\Lambda_n(t)\|\varphi_1\|_{H^{s_1}\cap L^1}.
	\end{align}
	Applying \eqref{W1-energy-bound} with $t$ replaced by $t-\tau$, using Minkowski's inequality in \eqref{Wave-mild-formula}, and differentiating the mild formula with respect to time prove \eqref{Wave-Energy} and \eqref{Wave-Time-Derivative}.

	We next demonstrate \eqref{Wave-Linfty} by considering two different dimensions separately.
	\begin{description}
		\item[Three space dimensions] The Kirchhoff formula (see, for example, \cite{John=1981}) provides
		\begin{align}\label{Kirchhoff-formulas}
			W_1(t,|\nabla|)\varphi_1(x)&=\frac{1}{4\pi t}\int_{|x-y|=t}\varphi_1(y)\,\mathrm{d}S_y,\notag\\
			W_0(t,|\nabla|)\varphi_0(x)&=\partial_t\left(\frac{1}{4\pi t}\int_{|x-y|=t}\varphi_0(y)\,\mathrm{d}S_y\right).
		\end{align}
		For $t\geqslant1$, Lemma~\ref{Lemma-Trace} yields
		\begin{align}\label{W1-Linfty-3D}
			\|W_1(t,|\nabla|)\varphi_1(\cdot)\|_{L^\infty}&\lesssim t^{-1}\sup_{x\in\mb{R}^3}\int_{|x-y|=t}|\varphi_1(y)|\,\mathrm{d}S_y\notag\\[0.3em]
			&\lesssim t^{-1}\|\varphi_1\|_{W^{1,1}}.
		\end{align}
		For $0\leqslant t\leqslant1$, the equivalent representation
		\begin{align*}
			W_1(t,|\nabla|)\varphi_1(x)=\frac{t}{4\pi}\int_{\mb{S}^2}\varphi_1(x+t\,\omega)\,\mathrm{d}\omega
		\end{align*}
		implies
		\begin{align*}
			\|W_1(t,|\nabla|)\varphi_1(\cdot)\|_{L^\infty}\lesssim\|\varphi_1\|_{L^\infty}.
		\end{align*}
		For another, a direct differentiation in \eqref{Kirchhoff-formulas} gives
		\begin{align*}
			W_0(t,|\nabla|)\varphi_0(x)=\frac{1}{4\pi t}\int_{|x-y|=t}\nabla\varphi_0(y)\cdot\frac{y-x}{t}\,\mathrm{d}S_y+\frac{1}{4\pi t^2}\int_{|x-y|=t}\varphi_0(y)\,\mathrm{d}S_y.
		\end{align*}
		Thus, Lemma~\ref{Lemma-Trace}, applied to $\nabla\varphi_0$ and $\varphi_0$, shows
		\begin{align}\label{W0-Linfty-3D}
			\|W_0(t,|\nabla|)\varphi_0(\cdot)\|_{L^\infty}\lesssim t^{-1}\|\varphi_0\|_{W^{2,1}}
		\end{align}
		for $t\geqslant1$. For $0\leqslant t\leqslant1$, the representation on the unit sphere indicates
		\begin{align*}
			\|W_0(t,|\nabla|)\varphi_0(\cdot)\|_{L^\infty}\lesssim\|\varphi_0\|_{W^{1,\infty}}.
		\end{align*}
		Combining \eqref{W1-Linfty-3D} and \eqref{W0-Linfty-3D} proves the homogeneous part of \eqref{Wave-Linfty} for $n=3$.
		\item[Two space dimensions] The Poisson formula (see, for example, \cite{John=1981}) reads
		\begin{align}\label{Poisson-W1-merged}
			W_1(t,|\nabla|)\varphi_1(x)=\frac{1}{2\pi}\int_{|x-y|<t}\frac{\varphi_1(y)}{\sqrt{t^2-|x-y|^2}}\,\mathrm{d}y.
		\end{align}
		For $0\leqslant t\leqslant2$, the change of variables $y=x+tz$ gives
		\begin{align*}
			W_1(t,|\nabla|)\varphi_1(x)=\frac{t}{2\pi}\int_{|z|<1}\frac{\varphi_1(x+tz)}{\sqrt{1-|z|^2}}\,\mathrm{d}z,
		\end{align*}
		and consequently,
		\begin{align*}
			\|W_1(t,|\nabla|)\varphi_1(\cdot)\|_{L^\infty}\lesssim\|\varphi_1\|_{L^\infty}.
		\end{align*}
		For $t\geqslant2$, we split the integral in \eqref{Poisson-W1-merged} into the interior region $|x-y|\leqslant t-1$ and the boundary layer $t-1\leqslant|x-y|<t$. With the help of
		\begin{align*}
			t^2-|x-y|^2\geqslant t^2-(t-1)^2\gtrsim t
		\end{align*}
		on the interior region, one obtains
		\begin{align*}
			\int_{|x-y|\leqslant t-1}\frac{|\varphi_1(y)|}{\sqrt{t^2-|x-y|^2}}\,\mathrm{d}y\lesssim t^{-\frac12}\|\varphi_1\|_{L^1}.
		\end{align*}
		For the boundary layer, polar coordinates centered at $x$ and Lemma~\ref{Lemma-Trace} shows
		\begin{align*}
			\int_{t-1\leqslant|x-y|<t}\frac{|\varphi_1(y)|}{\sqrt{t^2-|x-y|^2}}\,\mathrm{d}y&\lesssim\int_{t-1}^{t}\frac{1}{\sqrt{t^2-\rho^2}}\left(\int_{|x-y|=\rho}|\varphi_1(y)|\,\mathrm{d}S_y\right)\mathrm{d}\rho\notag\\
			&\lesssim\|\varphi_1\|_{W^{1,1}}\int_{t-1}^{t}\frac{1}{\sqrt{t^2-\rho^2}}\,\mathrm{d}\rho\\[0.3em]
			&\lesssim t^{-\frac12}\|\varphi_1\|_{W^{1,1}}.
		\end{align*}
		The summary of them implies
		\begin{align}\label{W1-Linfty-2D-merged}
			\|W_1(t,|\nabla|)\varphi_1(\cdot)\|_{L^\infty}\lesssim\langle t\rangle^{-\frac12}\|\varphi_1\|_{W^{1,1}\cap L^\infty}.
		\end{align}
		Differentiating the Poisson formula yields
		\begin{align}\label{Poisson-W0-merged}
			W_0(t,|\nabla|)\varphi_0(x)=\frac{1}{2\pi t}\int_{|x-y|<t}\frac{\varphi_0(y)+(y-x)\cdot\nabla\varphi_0(y)}{\sqrt{t^2-|x-y|^2}}\,\mathrm{d}y.
		\end{align}
		For $0\leqslant t\leqslant2$, the representation before the change of variables implies
		\begin{align*}
			\|W_0(t,|\nabla|)\varphi_0(\cdot)\|_{L^\infty}\lesssim\|\varphi_0\|_{W^{1,\infty}}.
		\end{align*}
		For $t\geqslant2$, we apply the preceding interior-boundary decomposition to both terms in \eqref{Poisson-W0-merged}. In the term containing $\nabla\varphi_0$, the factor $t^{-1}|x-y|$ is bounded by $1$. Therefore,
		\begin{align}\label{W0-Linfty-2D-merged}
			\|W_0(t,|\nabla|)\varphi_0(\cdot)\|_{L^\infty}\lesssim t^{-\frac12}\|\varphi_0\|_{W^{2,1}}.
		\end{align}
		Combining \eqref{W1-Linfty-2D-merged} and \eqref{W0-Linfty-2D-merged} justifies the homogeneous part of \eqref{Wave-Linfty} for $n=2$.
	\end{description}
	Eventually, applying the estimate for $W_1(t,|\nabla|)$ in the corresponding dimension with $t$ replaced by $t-\tau$ directly, and then using Minkowski's inequality in \eqref{Wave-mild-formula}, yields the Duhamel contribution in \eqref{Wave-Linfty}.
\end{proof}

\begin{remark}[Dimension-dependent features of the wave estimates]
	The factor $\Lambda_n(t)$ is caused exclusively by the low-frequency part of $W_1(t,|\nabla|)$. The radial integral in \eqref{Wave-low-frequency-integral} is uniformly bounded in three space dimensions, whereas in two space dimensions it grows like $\ln(\mathrm{e}+t)$. The resulting factor $\Lambda_2(t)$ is sharp in general with the $L^1$ integrable data and reflects an intrinsic low-frequency property of the two-dimensional free wave equation (see
	Remark~\ref{Remark-Optimal-Logarithmic-Growth}). On the other hand, the decay rate in \eqref{Wave-Linfty} comes from the geometry of the Poisson and Kirchhoff formulas. The additional derivative imposed on $\varphi_0$ is needed when differentiating the corresponding representation formula for $W_0(t,|\nabla|)\varphi_0$. 	
\end{remark}

\section{Global in-time existence for Nakao's problem}
 We now prove Theorem~\ref{Thm-01}. The main point is that the regularity needed for the self-map estimate is stronger than that required for the contraction. We therefore introduce a strong evolution space in which the nonlinear map is invariant and a weaker metric in which it becomes contractive. After establishing the corresponding nonlinear estimates, we complete the fixed point argument and prove the scattering statement in Corollary~\ref{Cor-Free-Wave}.

\subsection{Functional setting and solution map}\label{Subsection-Functional-Setting}

 We begin by introducing the solution map associated with Nakao's problem \eqref{Nakao-Problem}. Let
\begin{align*}
	\ml{N}_1[v]&:=u^{\lin}+u^{\nlin},\\
	\ml{N}_2[u]&:=v^{\lin}+v^{\nlin},
\end{align*}
where $u^{\lin}=u^{\lin}(t,x)$ and $v^{\lin}=v^{\lin}(t,x)$ are
\begin{align*}
	u^{\lin}(t,x)&:=E_0^u(t,x)\ast_{(x)}u_0(x)+E_1^u(t,x)\ast_{(x)}u_1(x),\\
	v^{\lin}(t,x)&:=E_0^v(t,x)\ast_{(x)}v_0(x)+E_1^v(t,x)\ast_{(x)}v_1(x),
\end{align*}
and the Duhamel parts $u^{\nlin}=u^{\nlin}(t,x)$ and $v^{\nlin}=v^{\nlin}(t,x)$ are defined by
\begin{align*}
	u^{\nlin}(t,x)&:=\int_0^tE_1^u(t-\tau,x)\ast_{(x)}|v(\tau,x)|^p\,\mathrm{d}\tau,\\
	v^{\nlin}(t,x)&:=\int_0^tE_1^v(t-\tau,x)\ast_{(x)}|u(\tau,x)|^q\,\mathrm{d}\tau.
\end{align*}
Here, $E_k^u=E_k^u(t,x)$ and $E_k^v=E_k^v(t,x)$ for $k\in\{0,1\}$ denote the fundamental solutions to the linear Cauchy problems \eqref{Damped-Wave} and \eqref{Wave}, respectively.

By Duhamel's principle, we define the operator
\begin{align*}
	\ml{N}:\ (u,v)\in Z_T\mapsto\ml{N}[u,v]:=\big(\ml{N}_1[v],\ml{N}_2[u]\big)
\end{align*}
on the evolution space $Z_T:=X_T\times Y_T$. Motivated by Lemma~\ref{Lemma-damped-wave} and Lemma~\ref{Lemma-Wave}, for any $T>0$, we introduce
\begin{align*}
	X_T&:=\ml{C}\big([0,T],W^{2,r_n}\cap H^2\big)\cap\ml{C}^1\big([0,T],H^1\big),\\
	Y_T&:=\ml{C}\big([0,T],L^\infty\cap H^2\big)\cap\ml{C}^1\big([0,T],H^1\big),
\end{align*}
endowed with the corresponding norms
\begin{align*}
	\|u\|_{X_T}&:=\sup\limits_{t\in[0,T]}\Big(\langle t\rangle^{\frac{n}{2}\left(\frac{1}{m}-\frac{1}{r_n}\right)}\|u(t,\cdot)\|_{W^{2,r_n}}+\|u(t,\cdot)\|_{H^2}+\langle t\rangle\|u_t(t,\cdot)\|_{H^1}\Big),\\
	\|v\|_{Y_T}&:=\sup\limits_{t\in[0,T]}\Big(\langle t\rangle^{\frac{n-1}{2}}\|v(t,\cdot)\|_{L^\infty}+[\Lambda_n(t)]^{-1}\|v(t,\cdot)\|_{H^2}+\|v_t(t,\cdot)\|_{H^1}\Big).
\end{align*}
For this reason, we set $\|(u,v)\|_{Z_T}:=\|u\|_{X_T}+\|v\|_{Y_T}$.

\subsection{Strategy of the fixed point argument}

We first construct a solution on an arbitrary finite interval $[0,T]$, with all estimates uniform with respect to $T>0$. Let
\begin{align}\label{Size-of-initial-data}
	\varepsilon:=\|(u_0,u_1)\|_{\ml{A}_n}+\|(v_0,v_1)\|_{\ml{B}_n}
\end{align}
denote the size of the initial data. The fixed point argument is carried out at two different levels. The nonlinear map is invariant in the strong solution space $Z_T$, whereas the contraction is established with respect to a weaker metric.

\medskip

\noindent\textit{Strong self-map.}
The first step is to prove
\begin{align}\label{Crucial-01}
	\|\ml{N}[u,v]\|_{Z_T}\lesssim\varepsilon+\|v\|_{Y_T}^p+\|u\|_{X_T}^q,
\end{align}
uniformly for $T>0$. Consequently, if
\begin{align}\label{Ball-RT}
	B_R(T):=\big\{(u,v)\in Z_T:\ \|(u,v)\|_{Z_T}\leqslant R\big\},
\end{align}
then by taking $R$ proportional to $\varepsilon$ and choosing $\varepsilon>0$ sufficiently small, we obtain
\begin{align}\label{Ball-subset}
	\ml{N}\big(B_R(T)\big)\subset B_R(T).
\end{align}

\medskip

\noindent\textit{Obstruction to strong contraction.}
A direct contraction argument in $Z_T$ would require a second-order difference estimate for $|v|^p-|\tilde v|^p$. At the level of the second-order derivatives, one encounters the difference
\begin{align*}
	\big||v|^{p-2}-|\tilde v|^{p-2}\big|.
\end{align*}
When $2<p<3$, the map $z\mapsto |z|^{p-2}$ is only H\"older continuous near the origin and does not satisfy the Lipschitz estimate required by a contraction argument in the strong norm. This would impose the artificial restriction $p\geqslant3$ in three space dimensions. We therefore lower the spatial regularity of the contraction metric by one derivative, while retaining the time-weighted $L^\infty$ components needed to estimate the nonlinear source terms. In this weaker metric, only first-order nonlinear differences are required. Although the two-dimensional admissible range already satisfies $p>3$, we use the same weak metric in both dimensions $n\in\{2,3\}$ in order to present a unified argument.

\medskip

\noindent\textit{Weak contraction.}
After introducing the metric $d_T$ in Section~\ref{Subsection-Weak-Contraction}, we prove
\begin{align}\label{Crucial-03}
	d_T\big(\ml{N}[w],\ml{N}[\tilde w]\big)\leqslant C(R^{p-1}+R^{q-1})\,d_T(w,\tilde w)
\end{align}
for all $w,\tilde w\in B_R(T)$, uniformly for $T>0$. By decreasing $R>0$ if necessary, the coefficient on the right-hand side of \eqref{Crucial-03} is bounded by $\frac12$.

\medskip

\noindent\textit{Recovery of the strong solution and passage to global time.} 
Since $B_R(T)$ is defined by the strong norm but the contraction is performed in a weaker metric, completeness does not follow directly from the strong self-map property. Instead, the Picard sequence is shown to converge in the weak evolution space, while remaining uniformly bounded in $B_R(T)$. Weak-$\ast$ compactness and lower semicontinuity then recover the strong bounds for the limit, and the Duhamel representation identifies it as a fixed point of $\ml{N}$ in $Z_T$. Finally, because all estimates are independent of $T$, the finite-time solutions are compatible by uniqueness and define a global in-time solution.

\subsection{Nonlinear estimates}

 The following Sobolev embeddings will be used repeatedly in this part:
\begin{align}\label{Embeddings-u}
	W^{2,r_n}&\hookrightarrow W^{1,q}\cap W^{1,2q}\cap L^\infty,\\
	\label{Embeddings-v}H^2&\hookrightarrow W^{2-\frac{2}{r_n},r_n}.
\end{align}
\begin{itemize}
	\item For $n=2$, these follow from $r_2>2$, $q\geqslant r_2$ and $r_2\leqslant4$.
	\item For $n=3$, the second embedding follows from $r_3\leqslant\frac{10}{3}$. If $2<q\leqslant\frac{10}{3}$, then $r_3=q$ and the first embedding is standard; if $q>\frac{10}{3}$, then $r_3=\frac{10}{3}>3$ and $W^{2,r_3}\hookrightarrow W^{1,\infty}$ implies the first embedding.
\end{itemize}
  Moreover, we shall use the following fractional power rule (see, for example, \cite{Palmieri-Reissig=2018}):
\begin{align*}
	\|\,|f|^p\|_{W^{s,r}}\lesssim\|f\|_{L^\infty}^{p-1}\|f\|_{W^{s,r}}\ \ \mbox{with}\ \ p>2
\end{align*}
for $f\in W^{s,r}\cap L^\infty$, where $1<s<2$ and $1<r<\infty$.

\begin{lemma}[Strong nonlinear estimates]\label{Lemma-Nonlinear-Strong}
	Let $u\in X_T$ and $v\in Y_T$. Then, the following estimates hold:
	\begin{align}\label{Nonlinear-v-W2m}
		\|\,|v(\tau,\cdot)|^p\|_{W^{2,m}}&\lesssim\langle\tau\rangle^{-\frac{n-1}{2}\left(p-\frac{2}{m}\right)}[\Lambda_n(\tau)]^{\frac{2}{m}}\|v\|_{Y_\tau}^p,\\
		\label{Nonlinear-v-high}
		\|\,|v(\tau,\cdot)|^p\|_{W^{2-\frac{2}{r_n},r_n}}+\|\,|v(\tau,\cdot)|^p\|_{H^1}&\lesssim\langle\tau\rangle^{-\frac{n-1}{2}(p-1)}\Lambda_n(\tau)\|v\|_{Y_\tau}^p,\\
		\label{Nonlinear-u-source}
		\|\,|u(\tau,\cdot)|^q\|_{W^{1,1}\cap L^\infty}+\|\,|u(\tau,\cdot)|^q\|_{H^1\cap L^1}&\lesssim\langle\tau\rangle^{-q\frac{n}{2}\left(\frac{1}{m}-\frac{1}{r_n}\right)}\|u\|_{X_\tau}^q,
	\end{align}
	for all $\tau\in[0,T]$.
\end{lemma}

\begin{proof}
	We first estimate the source term $|v(\tau,x)|^p$. The chain rule implies
	\begin{align}\label{Chain-v-W2m}
		\|\,|v(\tau,\cdot)|^p\|_{W^{2,m}}&\lesssim\|v(\tau,\cdot)\|_{L^{mp}}^p+\|\,|v(\tau,\cdot)|^{p-1}|\nabla v(\tau,\cdot)|\,\|_{L^m}\notag\\
		&\quad\ +\|\,|v(\tau,\cdot)|^{p-1}|\nabla^2v(\tau,\cdot)|\,\|_{L^m}+\|\,|v(\tau,\cdot)|^{p-2}|\nabla v(\tau,\cdot)|^2\|_{L^m}.
	\end{align}
	Interpolation between $L^2$ and $L^\infty$ gives
	\begin{align*}
		\|v(\tau,\cdot)\|_{L^{mp}}^p\lesssim\|v(\tau,\cdot)\|_{L^\infty}^{p-\frac2m}\|v(\tau,\cdot)\|_{L^2}^{\frac2m}.
	\end{align*}
	Moreover, by H\"older's inequality with the exponent $\frac{2m(p-1)}{2-m}$, we deduce
	\begin{align*}
		&\|\,|v(\tau,\cdot)|^{p-1}|\nabla v(\tau,\cdot)|\,\|_{L^m}+\|\,|v(\tau,\cdot)|^{p-1}|\nabla^2v(\tau,\cdot)|\,\|_{L^m}\\[0.3em]
		&\qquad\lesssim\|v(\tau,\cdot)\|_{L^{\frac{2m(p-1)}{2-m}}}^{p-1}\big(\|\nabla v(\tau,\cdot)\|_{L^2}+\|\nabla^2v(\tau,\cdot)\|_{L^2}\big)\\
		&\qquad\lesssim\|v(\tau,\cdot)\|_{L^\infty}^{p-\frac2m}\|v(\tau,\cdot)\|_{H^2}^{\frac2m}.
	\end{align*}
	To estimate the last term in \eqref{Chain-v-W2m}, we use $2<2m<4$. Interpolation between $L^2$ and $L^4$, together with the Gagliardo-Nirenberg inequality, yields
	\begin{align*}
		\|\nabla v(\tau,\cdot)\|_{L^{2m}}^2&\lesssim\|\nabla v(\tau,\cdot)\|_{L^2}^{\frac4m-2}\|\nabla v(\tau,\cdot)\|_{L^4}^{4-\frac4m}\\
        &\lesssim\|\nabla v(\tau,\cdot)\|_{L^2}^{\frac4m-2}\|v(\tau,\cdot)\|_{L^\infty}^{2-\frac2m}\|\nabla^2v(\tau,\cdot)\|_{L^2}^{2-\frac2m}\\
		&\lesssim\|v(\tau,\cdot)\|_{L^\infty}^{2-\frac2m}\|v(\tau,\cdot)\|_{H^2}^{\frac2m}.
	\end{align*}
	Consequently,
	\begin{align*}
		\|\,|v(\tau,\cdot)|^{p-2}|\nabla v(\tau,\cdot)|^2\|_{L^m}\lesssim\|v(\tau,\cdot)\|_{L^\infty}^{p-\frac2m}\|v(\tau,\cdot)\|_{H^2}^{\frac2m}.
	\end{align*}
	Combining these estimates with the definition of $Y_\tau$, we arrive at \eqref{Nonlinear-v-W2m}.

	Next, notice that $1<2-\frac2{r_n}<2$. Due to $p>2$, the fractional chain rule and \eqref{Embeddings-v} imply
	\begin{align*}
		\|\,|v(\tau,\cdot)|^p\|_{W^{2-\frac2{r_n},r_n}}&\lesssim\|v(\tau,\cdot)\|_{L^\infty}^{p-1}\|v(\tau,\cdot)\|_{W^{2-\frac2{r_n},r_n}}\\
		&\lesssim\|v(\tau,\cdot)\|_{L^\infty}^{p-1}\|v(\tau,\cdot)\|_{H^2}\\
		&\lesssim \langle\tau\rangle^{-\frac{n-1}{2}(p-1)}\Lambda_n(\tau)\|v\|_{Y_\tau}^p.
	\end{align*}
	The chain rule gives the same bound in $H^1$. Therefore,  \eqref{Nonlinear-v-high} is valid.

	Finally, by \eqref{Embeddings-u} and the definition of $X_\tau$, we obtain
	\begin{align}\label{u-strong-decay}
		\|u(\tau,\cdot)\|_{W^{1,q}\cap W^{1,2q}\cap L^\infty}\lesssim\langle\tau\rangle^{-\frac n2\left(\frac1m-\frac1{r_n}\right)}\|u\|_{X_\tau}.
	\end{align}
	Using H\"older's inequality and the chain rule, one has
	\begin{align*}
		\|\,|u(\tau,\cdot)|^q\|_{W^{1,1}}&\lesssim\|u(\tau,\cdot)\|_{L^q}^q+\|u(\tau,\cdot)\|_{L^q}^{q-1}\|\nabla u(\tau,\cdot)\|_{L^q},\\
		\|\,|u(\tau,\cdot)|^q\|_{H^1}&\lesssim\|u(\tau,\cdot)\|_{L^{2q}}^q+\|u(\tau,\cdot)\|_{L^{2q}}^{q-1}\|\nabla u(\tau,\cdot)\|_{L^{2q}},\\
		\|\,|u(\tau,\cdot)|^q\|_{L^\infty}&\lesssim\|u(\tau,\cdot)\|_{L^\infty}^q.
	\end{align*}
	The $L^1$ estimate is contained in the $W^{1,1}$ estimate. Substituting \eqref{u-strong-decay} proves \eqref{Nonlinear-u-source}. 
\end{proof}

\begin{remark}[Regularity of the nonlinear compositions]
	The condition $p>2$ is sufficient for all estimates in Lemma~\ref{Lemma-Nonlinear-Strong}. 
	\begin{itemize}
		\item For the integer-order $W^{2,m}$ estimate, the chain rule produces the terms $|v|^{p-1}\nabla^2v$ and $|v|^{p-2}|\nabla v|^2$. Both are well-defined and can be estimated for a single function when $p>2$.
		\item For the Bessel-potential estimate in $W^{2-\frac2{r_n},r_n}$, the differentiability order lies strictly between $1$ and $2$, and the fractional chain rule is again available under $p>2$.
	\end{itemize}   The stronger condition $p\geqslant3$ would be needed only for a Lipschitz estimate of the second-order difference $|v|^p-|\tilde v|^p$ in the strong contraction norm, because this would require Lipschitz continuity of the map $z\mapsto|z|^{p-2}$. This is the artificial restriction avoided by the weak metric.
\end{remark}

\subsection{Strong self-map property}

Recalling that
\begin{align}\label{Condition-N-1}
	(n-1)\left(\frac12-\frac1{r_n}\right)\leqslant1-\frac2{r_n}\ \ \mbox{for}\ \ n\in\{2,3\},
\end{align}
the high-frequency regularity required in Lemma~\ref{Lemma-damped-wave} is contained in the definition of $\ml{A}_n$. Therefore, by Lemma~\ref{Lemma-damped-wave} and Lemma~\ref{Lemma-Wave}, the linear parts satisfy
\begin{align}\label{Linear-parts-estimate}
	\|u^{\lin}\|_{X_T}+\|v^{\lin}\|_{Y_T}\lesssim\|(u_0,u_1)\|_{\ml{A}_n}+\|(v_0,v_1)\|_{\ml{B}_n}.
\end{align}
We now estimate the Duhamel parts. Applying Lemma~\ref{Lemma-damped-wave} with $(p_0,q_0)=(m,r_n)$ and $s_0=1$, we obtain
\begin{align*}
	\|u^{\nlin}(t,\cdot)\|_{W^{2,r_n}}&\lesssim\int_0^t\langle t-\tau\rangle^{-\frac n2\left(\frac1m-\frac1{r_n}\right)}\|\,|v(\tau,\cdot)|^p\|_{W^{2,m}}\,\mathrm{d}\tau\\
	&\quad\ +\int_0^t\mathrm{e}^{-c(t-\tau)}\|\,|v(\tau,\cdot)|^p\|_{W^{1+(n-1)\left(\frac12-\frac1{r_n}\right),r_n}}\,\mathrm{d}\tau.
\end{align*}
Thanks to \eqref{Condition-N-1}, the second norm is controlled by the left-hand side of \eqref{Nonlinear-v-high}. Using \eqref{Nonlinear-v-W2m}, \eqref{Nonlinear-v-high}, the inequalities in \eqref{Time-integrability-condition} and Lemma~\ref{Lemma-Time-Convolution}, it follows that
\begin{align*}
	\|u^{\nlin}(t,\cdot)\|_{W^{2,r_n}}\lesssim\langle t\rangle^{-\frac n2\left(\frac1m-\frac1{r_n}\right)}\|v\|_{Y_T}^p.
\end{align*}
For the energy components, Lemma~\ref{Lemma-damped-wave} and \eqref{Nonlinear-v-high} give
\begin{align*}
	\|u^{\nlin}(t,\cdot)\|_{H^2}&\lesssim\int_0^t\langle\tau\rangle^{-\frac{n-1}{2}(p-1)}\Lambda_n(\tau)\,\mathrm{d}\tau\,\|v\|_{Y_T}^p\lesssim\|v\|_{Y_T}^p,\\\|u_t^{\nlin}(t,\cdot)\|_{H^1}
	&\lesssim\int_0^t\langle t-\tau\rangle^{-1}\langle\tau\rangle^{-\frac{n-1}{2}(p-1)}\Lambda_n(\tau)\,\mathrm{d}\tau\,\|v\|_{Y_T}^p\lesssim\langle t\rangle^{-1}\|v\|_{Y_T}^p.
\end{align*}
Consequently,
\begin{align}\label{Strong-N1}
	\|\ml{N}_1[v]\|_{X_T}\lesssim\|(u_0,u_1)\|_{\ml{A}_n}+\|v\|_{Y_T}^p.
\end{align}

For the undamped component, Lemma~\ref{Lemma-Wave} and \eqref{Nonlinear-u-source} yield
\begin{align*}
	\|v^{\nlin}(t,\cdot)\|_{L^\infty}\lesssim\int_0^t\langle t-\tau\rangle^{-\frac{n-1}{2}}\langle\tau\rangle^{-q\frac n2\left(\frac1m-\frac1{r_n}\right)}\mathrm{d}\tau\,\|u\|_{X_T}^q\lesssim\langle t\rangle^{-\frac{n-1}{2}}\|u\|_{X_T}^q,
\end{align*}
where we used \eqref{Time-integrability-condition}. Moreover, we find
\begin{align*}
	\|v^{\nlin}(t,\cdot)\|_{H^2}\lesssim\int_0^t\Lambda_n(t-\tau)\langle\tau\rangle^{-q\frac n2\left(\frac1m-\frac1{r_n}\right)}\mathrm{d}\tau\,\|u\|_{X_T}^q\lesssim\Lambda_n(t)\|u\|_{X_T}^q,
\end{align*}
where we used the monotonicity of $\Lambda_n$ and the integrability of the time weight. Similarly,
\begin{align*}
	\|v_t^{\nlin}(t,\cdot)\|_{H^1}\lesssim\int_0^t\langle\tau\rangle^{-q\frac n2\left(\frac1m-\frac1{r_n}\right)}\mathrm{d}\tau\,\|u\|_{X_T}^q\lesssim\|u\|_{X_T}^q.
\end{align*}
Consequently,
\begin{align}\label{Strong-N2}
	\|\ml{N}_2[u]\|_{Y_T}\lesssim\|(v_0,v_1)\|_{\ml{B}_n}+\|u\|_{X_T}^q.
\end{align}
Combining \eqref{Linear-parts-estimate}, \eqref{Strong-N1} and \eqref{Strong-N2}, we obtain \eqref{Crucial-01}.
Then, for $R=2C\varepsilon$ and $\varepsilon>0$ sufficiently small, the operator $\ml{N}$ maps $B_R(T)$ defined in \eqref{Ball-RT} into itself, uniformly with respect to $T>0$. This proves \eqref{Ball-subset}.

\subsection{Weak contraction estimate}\label{Subsection-Weak-Contraction}

 To avoid second-order differences of the nonlinearity $|v|^p$, we lower the spatial regularity of the contraction metric by one derivative. At the same time, we retain the weighted $L^\infty$ component required to estimate the difference of the source terms in the wave equation. Accordingly, we introduce the weak spaces
\begin{align*}
	\mb{X}_T&:=\ml{C}\big([0,T],W^{1,r_n}\cap L^\infty\cap H^1\big)\cap\ml{C}^1\big([0,T],L^2\big),\\
	\mb{Y}_T&:=\ml{C}\big([0,T],L^\infty\cap H^1\big)\cap\ml{C}^1\big([0,T],L^2\big),
\end{align*}
with the corresponding norms
\begin{align}\label{Weak-Norm-X}
	\|u\|_{\mb{X}_T}&:=\sup_{t\in[0,T]}\Big(\langle t\rangle^{\frac{n}{2}\left(\frac{1}{m}-\frac{1}{r_n}\right)}\|u(t,\cdot)\|_{W^{1,r_n}}+\langle t\rangle^{\frac{n}{2}\left(\frac{1}{m}-\frac{1}{r_n}\right)}\|u(t,\cdot)\|_{L^\infty}\notag\\
	&\qquad\qquad\ \ \, +\|u(t,\cdot)\|_{H^1}+\langle t\rangle\|u_t(t,\cdot)\|_{L^2}\Big),
\end{align}
	and
	\begin{align}
	 \label{Weak-Norm-Y}
	\|v\|_{\mb{Y}_T}&:=\sup_{t\in[0,T]}\Big(\langle t\rangle^{\frac{n-1}{2}}\|v(t,\cdot)\|_{L^\infty}+[\Lambda_n(t)]^{-1}\|v(t,\cdot)\|_{H^1}+\|v_t(t,\cdot)\|_{L^2}\Big).
\end{align}
Equipping the norms in \eqref{Weak-Norm-X} and \eqref{Weak-Norm-Y}, for $w=(u,v)$ and $\tilde w=(\tilde u,\tilde v)$, set
\begin{align*}
	d_T(w,\tilde w):=\|u-\tilde u\|_{\mb{X}_T}+\|v-\tilde v\|_{\mb{Y}_T}.
\end{align*}
Notice that $B_R(T)\subset\mb{X}_T\times\mb{Y}_T$ and $d_T$ is finite on $B_R(T)$.

For later use, we verify the completeness of the weak evolution space. Define
\begin{align*}
	\mb{E}_X&:=W^{1,r_n}\cap L^\infty\cap H^1,\\
	\mb{E}_Y&:=L^\infty\cap H^1,
\end{align*}
where each intersection space is endowed with the sum of the corresponding norms. Since all the component spaces are continuously embedded into $\ml{S}'$, both $\mb{E}_X$ and $\mb{E}_Y$ are Banach spaces.

For each fixed $T>0$, the time weights appearing in
\eqref{Weak-Norm-X} and \eqref{Weak-Norm-Y} are bounded from above
and below by positive constants on $[0,T]$. Hence,
$\|\cdot\|_{\mb{X}_T}$ and $\|\cdot\|_{\mb{Y}_T}$ are equivalent,
respectively, to the graph norms
\begin{align*}
	\|u\|_{\ml{C}([0,T],\mb{E}_X)}+\|u_t\|_{\ml{C}([0,T],L^2)} \ \ \mbox{and}\ \ \|v\|_{\ml{C}([0,T],\mb{E}_Y)}+\|v_t\|_{\ml{C}([0,T],L^2)}.
\end{align*}

Indeed, if $\{u^{(k)}\}_{k\in\mb{N}}$ is a Cauchy sequence in
$\mb{X}_T$, then for some $u$ and $g$,
\begin{align*}
	u^{(k)} \to u \ \ &\mbox{in} \ \ 	\ml{C}([0,T],\mb{E}_X),\\
	u_t^{(k)}\to g\ \ &\mbox{in}\ \ 	\ml{C}([0,T],L^2).
\end{align*}
According to
\begin{align*}
	u^{(k)}(t,\cdot)-u^{(k)}(0,\cdot)=\int_0^t u_t^{(k)}(\tau,\cdot)\,\mathrm{d}\tau \ \ \mbox{in}\ \ L^2,
\end{align*}
 passing to the limit gives
\begin{align*}
	u(t,\cdot)-u(0,\cdot)=\int_0^t g(\tau,\cdot)\,\mathrm{d}\tau.
\end{align*}
Thus, $u\in\ml{C}^1([0,T],L^2)$ and $u_t=g$. The same argument applies to $\mb{Y}_T$. Consequently, $\mb{X}_T\times\mb{Y}_T$ is a Banach space, and $d_T$ is the metric induced by its product norm.

\begin{remark}[The weighted $L^\infty$ component in $\mb{X}_T$]
	The weighted $L^\infty$ term in \eqref{Weak-Norm-X} is needed to control the difference of the source terms in the $L^\infty$ norm required by the Poisson- and Kirchhoff-type estimate \eqref{Wave-Linfty}. Indeed, for $w,\tilde w\in B_R(T)$, one observes
	\begin{align*}
		\|\,|u(\tau,\cdot)|^q-|\tilde u(\tau,\cdot)|^q\|_{L^\infty}&\lesssim\big(\|u(\tau,\cdot)\|_{L^\infty}+\|\tilde u(\tau,\cdot)\|_{L^\infty}\big)^{q-1}\|u(\tau,\cdot)-\tilde u(\tau,\cdot)\|_{L^\infty}\\
		&\lesssim R^{q-1}\langle\tau\rangle^{-q\frac n2\left(\frac1m-\frac1{r_n}\right)}d_T(w,\tilde w).
	\end{align*}
	Thus, the weight in \eqref{Weak-Norm-X} supplies the full time-decay rate required for the time-integrability condition in \eqref{Time-integrability-condition}. 
	\begin{itemize}
		\item For $n=2$, the $L^\infty$ component is already controlled by $W^{1,r_2}$ because $r_2>2$.
		\item For $n=3$ and $q>3$, one has $r_3>3$, so it is likewise controlled by $W^{1,r_3}$. However, the admissible three-dimensional range also contains $2<q\leqslant3$, for which $r_3=q$ and $W^{1,r_3}$ does not embed into $L^\infty$ in $\mb{R}^3$. The $L^\infty$ term is therefore included explicitly in order to treat the entire admissible range in a unified weak metric.
	\end{itemize}   
\end{remark}

\begin{remark}[Strong solution class and weak contraction metric]
	The weak metric is used only to compare two iterates. The self-map estimate keeps every Picard iterate uniformly bounded in the strong ball $B_R(T)$. Therefore, the fixed point obtained from the weak contraction still belongs to the strong solution class after the weak limit is identified by the Banach-Alaoglu theorem and lower semicontinuity.
\end{remark}

We next derive the nonlinear difference estimates required for the contraction argument. The coefficients in these estimates are controlled by the uniform strong bounds of the two elements in $B_R(T)$, whereas their difference is measured in the weaker metric $d_T$. Since the contraction norm is lower by one spatial derivative, only first-order differences of the nonlinearities are needed.

\begin{lemma}[Difference estimates]\label{Lemma-Difference}
	Let $w,\tilde w\in B_R(T)$. Then, the following estimates hold:
	\begin{align}\label{Difference-v-W1m}
		\|\,|v(\tau,\cdot)|^p-|\tilde v(\tau,\cdot)|^p\|_{W^{1,m}}&\lesssim R^{p-1}\langle\tau\rangle^{-\frac{n-1}{2}\left(p-\frac{2}{m}\right)}[\Lambda_n(\tau)]^{\frac{2}{m}}d_T(w,\tilde w),\\
		\label{Difference-v-high}
		\|\,|v(\tau,\cdot)|^p-|\tilde v(\tau,\cdot)|^p\|_{H^1}&\lesssim R^{p-1}\langle\tau\rangle^{-\frac{n-1}{2}(p-1)}\Lambda_n(\tau)\,d_T(w,\tilde w),\\
		\label{Difference-u}
		\|\,|u(\tau,\cdot)|^q-|\tilde u(\tau,\cdot)|^q\|_{W^{1,1}\cap L^\infty\cap L^2}&\lesssim R^{q-1}\langle\tau\rangle^{-q\frac{n}{2}\left(\frac{1}{m}-\frac{1}{r_n}\right)}d_T(w,\tilde w),
	\end{align}
	for all $\tau\in[0,T]$.
\end{lemma}

\begin{proof}
	Let $F_p(\psi):=|\psi|^p$. Since $p>2$, for every $(\tau,x)\in[0,T]\times\mb{R}^n$, one has
	\begin{align}\label{Fp-difference-pointwise}
		|F_p(v(\tau,x))-F_p(\tilde v(\tau,x))|&\lesssim\big(|v(\tau,x)|+|\tilde v(\tau,x)|\big)^{p-1}|v(\tau,x)-\tilde v(\tau,x)|,\notag\\
		|F_p'(v(\tau,x))-F_p'(\tilde v(\tau,x))|&\lesssim\big(|v(\tau,x)|+|\tilde v(\tau,x)|\big)^{p-2}|v(\tau,x)-\tilde v(\tau,x)|.
	\end{align}
	For the zero-order term in $L^m$, H\"older's inequality gives
	\begin{align*}
		&\|F_p(v(\tau,\cdot))-F_p(\tilde v(\tau,\cdot))\|_{L^m}\\
		&\qquad\lesssim\left(\|v(\tau,\cdot)\|_{L^{\frac{2m(p-1)}{2-m}}}+\|\tilde v(\tau,\cdot)\|_{L^{\frac{2m(p-1)}{2-m}}}\right)^{p-1}\|v(\tau,\cdot)-\tilde v(\tau,\cdot)\|_{L^2}.
	\end{align*}
	Interpolating between $L^2$ and $L^\infty$, and using the definitions of $Y_T$ and $\mb{Y}_T$, we derive
	\begin{align}\label{Fp-zero-W1m}
		\|F_p(v(\tau,\cdot))-F_p(\tilde v(\tau,\cdot))\|_{L^m}\lesssim R^{p-1}\langle\tau\rangle^{-\frac{n-1}{2}\left(p-\frac2m\right)}[\Lambda_n(\tau)]^{\frac2m}d_T(w,\tilde w).
	\end{align}
	For the gradient, we write
	\begin{align}\label{Fp-gradient-difference}
		\nabla\big(F_p(v(\tau,x))-F_p(\tilde v(\tau,x))\big)&=F_p'(v(\tau,x))\nabla\big(v(\tau,x)-\tilde v(\tau,x)\big)\notag\\
		&\quad\ +\big(F_p'(v(\tau,x))-F_p'(\tilde v(\tau,x))\big)\nabla\tilde v(\tau,x).
	\end{align}
	The first term is estimated exactly as in \eqref{Fp-zero-W1m}, with the $L^2$ norm of the difference replaced by its $H^1$ norm. For the second term, H\"older's inequality with the exponent $\frac{2m(p-2)}{2-m}$ suggests
	\begin{align*}
		&\big\|\big(F_p'(v(\tau,\cdot))-F_p'(\tilde v(\tau,\cdot))\big)\nabla\tilde v(\tau,\cdot)\big\|_{L^m}\\
		&\qquad\lesssim\left(\|v(\tau,\cdot)\|_{L^{\frac{2m(p-2)}{2-m}}}+\|\tilde v(\tau,\cdot)\|_{L^{\frac{2m(p-2)}{2-m}}}\right)^{p-2}\|v(\tau,\cdot)-\tilde v(\tau,\cdot)\|_{L^\infty}\|\nabla\tilde v(\tau,\cdot)\|_{L^2}.
	\end{align*}
	The exponent is at least $2$ because the assumptions on $m$ imply $m>\frac{2}{p-1}$. Interpolation between $L^2$ and $L^\infty$ therefore gives the same right-hand side as in \eqref{Fp-zero-W1m}. This proves \eqref{Difference-v-W1m}.

	The $H^1$ estimate is simpler. Indeed, \eqref{Fp-difference-pointwise} and \eqref{Fp-gradient-difference} imply
	\begin{align*}
		&\|F_p(v(\tau,\cdot))-F_p(\tilde v(\tau,\cdot))\|_{H^1}\\
		&\qquad\lesssim\big(\|v(\tau,\cdot)\|_{L^\infty}+\|\tilde v(\tau,\cdot)\|_{L^\infty}\big)^{p-1}\|v(\tau,\cdot)-\tilde v(\tau,\cdot)\|_{H^1}\\
		&\qquad\quad\ +\big(\|v(\tau,\cdot)\|_{L^\infty}+\|\tilde v(\tau,\cdot)\|_{L^\infty}\big)^{p-2}\|v(\tau,\cdot)-\tilde v(\tau,\cdot)\|_{L^\infty}\|\tilde v(\tau,\cdot)\|_{H^1}.
	\end{align*}
	Using the strong and weak time weights proves \eqref{Difference-v-high}.

	We now consider $G_q(\psi):=|\psi|^q$. The mean value theorem indicates
	\begin{align}\label{Gq-difference-pointwise}
		|G_q(u(\tau,x))-G_q(\tilde u(\tau,x))|&\lesssim\big(|u(\tau,x)|+|\tilde u(\tau,x)|\big)^{q-1}|u(\tau,x)-\tilde u(\tau,x)|,\notag\\
		|G_q'(u(\tau,x))-G_q'(\tilde u(\tau,x))|&\lesssim\big(|u(\tau,x)|+|\tilde u(\tau,x)|\big)^{q-2}|u(\tau,x)-\tilde u(\tau,x)|.
	\end{align}
	By the pointwise estimates in \eqref{Gq-difference-pointwise}, since $q\geqslant r_n$ and $W^{1,r_n}\cap L^\infty\hookrightarrow L^s$ for every $s\in[r_n,\infty]$ occurring below, the zero-order estimates follow from
	\begin{align*}
		\|G_q(u(\tau,\cdot))-G_q(\tilde u(\tau,\cdot))\|_{L^1}&\lesssim\big(\|u(\tau,\cdot)\|_{L^q}+\|\tilde u(\tau,\cdot)\|_{L^q}\big)^{q-1}\|u(\tau,\cdot)-\tilde u(\tau,\cdot)\|_{L^q},\\
		\|G_q(u(\tau,\cdot))-G_q(\tilde u(\tau,\cdot))\|_{L^2}&\lesssim\big(\|u(\tau,\cdot)\|_{L^{2q}}+\|\tilde u(\tau,\cdot)\|_{L^{2q}}\big)^{q-1}\|u(\tau,\cdot)-\tilde u(\tau,\cdot)\|_{L^{2q}},\\
		\|G_q(u(\tau,\cdot))-G_q(\tilde u(\tau,\cdot))\|_{L^\infty}&\lesssim\big(\|u(\tau,\cdot)\|_{L^\infty}+\|\tilde u(\tau,\cdot)\|_{L^\infty}\big)^{q-1}\|u(\tau,\cdot)-\tilde u(\tau,\cdot)\|_{L^\infty}.
	\end{align*}
	For the gradient, we use
	\begin{align*}
		\nabla\big(G_q(u(\tau,x))-G_q(\tilde u(\tau,x))\big)&=G_q'(u(\tau,x))\nabla\big(u(\tau,x)-\tilde u(\tau,x)\big)\\
		&\quad\ +\big(G_q'(u(\tau,x))-G_q'(\tilde u(\tau,x))\big)\nabla\tilde u(\tau,x).
	\end{align*}
	H\"older's inequality derives
	\begin{align*}
		&\big\|\big(|u(\tau,\cdot)|+|\tilde u(\tau,\cdot)|\big)^{q-1}\nabla\big(u(\tau,\cdot)-\tilde u(\tau,\cdot)\big)\big\|_{L^1}\\
		&\qquad\lesssim\left(\|u(\tau,\cdot)\|_{L^{\frac{r_n(q-1)}{r_n-1}}}+\|\tilde u(\tau,\cdot)\|_{L^{\frac{r_n(q-1)}{r_n-1}}}\right)^{q-1}\big\|\nabla\big(u(\tau,\cdot)-\tilde u(\tau,\cdot)\big)\big\|_{L^{r_n}},
			\end{align*}
			and
		\begin{align*}
		&\big\|\big(|u(\tau,\cdot)|+|\tilde u(\tau,\cdot)|\big)^{q-2}|u(\tau,\cdot)-\tilde u(\tau,\cdot)|\,|\nabla\tilde u(\tau,\cdot)|\,\big\|_{L^1}\\
		&\qquad\lesssim\left(\|u(\tau,\cdot)\|_{L^{\frac{qr_n(q-2)}{qr_n-q-r_n}}}+\|\tilde u(\tau,\cdot)\|_{L^{\frac{qr_n(q-2)}{qr_n-q-r_n}}}\right)^{q-2}\|u(\tau,\cdot)-\tilde u(\tau,\cdot)\|_{L^q}\|\nabla\tilde u(\tau,\cdot)\|_{L^{r_n}}.
	\end{align*}
	Due to $q\geqslant r_n$, both Lebesgue exponents are at least $r_n$. Hence, the required bounds follow from $W^{1,r_n}\cap L^\infty\hookrightarrow L^s$ for every $s\in[r_n,\infty]$. This proves \eqref{Difference-u}.
\end{proof}

To estimate the high-frequency part of the damped wave difference, we use
\begin{align}\label{Difference-high-embedding}
	H^1\hookrightarrow W^{1-\frac{2}{r_n},r_n}\hookrightarrow W^{(n-1)\left(\frac12-\frac1{r_n}\right),r_n}.
\end{align}
For $n=2$, the first embedding follows from $r_2\leqslant4$, while the second one follows from $r_2>2$. For $n=3$, the second embedding is an equality of differentiability orders, and the first one follows from $r_3\leqslant\frac{10}{3}$. Thus, \eqref{Difference-high-embedding} provides precisely the high-frequency regularity required by the lower-order version of Lemma~\ref{Lemma-damped-wave}.

Applying the lower-order version of Lemma~\ref{Lemma-damped-wave} with $(p_0,q_0)=(m,r_n)$, and then using Lemma~\ref{Lemma-Difference}, \eqref{Difference-high-embedding} and Lemma~\ref{Lemma-Time-Convolution}, we obtain
\begin{align*}
	&\|\ml{N}_1[v](t,\cdot)-\ml{N}_1[\tilde v](t,\cdot)\|_{W^{1,r_n}}\\
	&\qquad \lesssim R^{p-1}d_T(w,\tilde w)\int_0^t\langle t-\tau\rangle^{-\frac n2\left(\frac1m-\frac1{r_n}\right)}\langle\tau\rangle^{-\frac{n-1}{2}\left(p-\frac2m\right)}[\Lambda_n(\tau)]^{\frac2m}\,\mathrm{d}\tau\\
	&\qquad\quad\ +R^{p-1}d_T(w,\tilde w)\int_0^t\mathrm{e}^{-c(t-\tau)}\langle\tau\rangle^{-\frac{n-1}{2}(p-1)}\Lambda_n(\tau)\,\mathrm{d}\tau\\
	&\qquad\lesssim\langle t\rangle^{-\frac n2\left(\frac1m-\frac1{r_n}\right)}R^{p-1}d_T(w,\tilde w).
\end{align*}
Analogously,
\begin{align*}
	\|\ml{N}_1[v](t,\cdot)-\ml{N}_1[\tilde v](t,\cdot)\|_{H^1}+\langle t\rangle\|\partial_t\ml{N}_1[v](t,\cdot)-\partial_t\ml{N}_1[\tilde v](t,\cdot)\|_{L^2}\lesssim R^{p-1}d_T(w,\tilde w).
\end{align*}

It remains to control the $L^\infty$ component in \eqref{Weak-Norm-X}. We know the following facts.
\begin{itemize}
	\item If $n=2$, the estimate follows directly from $W^{1,r_2}\hookrightarrow L^\infty$.
	\item If $n=3$ and $q>3$, then $r_3>3$, and the estimate follows from $W^{1,r_3}\hookrightarrow L^\infty$. Let now $n=3$ and $2<q\leqslant3$. Choose $3<\sigma\leqslant\frac{10}{3}$ sufficiently close to $3$ so that $m>\frac{3\sigma}{2\sigma+3}$.
	Such a choice is possible because $m>1$ and $\frac{3\sigma}{2\sigma+3}\downarrow1$ as $\sigma\downarrow3$.
\end{itemize}
	Since $H^1\hookrightarrow W^{1-\frac2\sigma,\sigma}$ and $W^{1,\sigma}\hookrightarrow L^\infty$, the lower-order damped wave estimate with target exponent $\sigma$ gives
\begin{align*}
	&\|\ml{N}_1[v](t,\cdot)-\ml{N}_1[\tilde v](t,\cdot)\|_{L^\infty}\lesssim\langle t\rangle^{-\frac32\left(\frac1m-\frac1\sigma\right)}
	R^{p-1}d_T(w,\tilde w).
\end{align*}
Because $\sigma>q$, one has 
\begin{align*}
\frac32\left(\frac1m-\frac1\sigma\right)
>\frac32\left(\frac1m-\frac1q\right),
\end{align*}
and hence this decay is stronger than the weight required in \eqref{Weak-Norm-X}. In other words,
\begin{align}\label{Weak-N1}
	\|\ml{N}_1[v]-\ml{N}_1[\tilde v]\|_{\mb{X}_T}\lesssim R^{p-1}d_T(w,\tilde w).
\end{align}

For the wave component, Lemma~\ref{Lemma-Wave} and \eqref{Difference-u} imply
\begin{align*}
	\|\ml{N}_2[u](t,\cdot)-\ml{N}_2[\tilde u](t,\cdot)\|_{L^\infty}&\lesssim R^{q-1}d_T(w,\tilde w)\int_0^t\langle t-\tau\rangle^{-\frac{n-1}{2}}\langle\tau\rangle^{-q\frac n2\left(\frac1m-\frac1{r_n}\right)}\,\mathrm{d}\tau\\
	&\lesssim\langle t\rangle^{-\frac{n-1}{2}}R^{q-1}d_T(w,\tilde w).
\end{align*}
Moreover, we find
\begin{align*}
	\|\ml{N}_2[u](t,\cdot)-\ml{N}_2[\tilde u](t,\cdot)\|_{H^1}&\lesssim R^{q-1}d_T(w,\tilde w)\int_0^t\Lambda_n(t-\tau)\langle\tau\rangle^{-q\frac n2\left(\frac1m-\frac1{r_n}\right)}\,\mathrm{d}\tau\\
	&\lesssim\Lambda_n(t)R^{q-1}d_T(w,\tilde w),
\end{align*}
and the corresponding time derivative is bounded in $L^2$ without the factor $\Lambda_n(t)$. One claims
\begin{align}\label{Weak-N2}
	\|\ml{N}_2[u]-\ml{N}_2[\tilde u]\|_{\mb{Y}_T}\lesssim R^{q-1}d_T(w,\tilde w).
\end{align}
Combining \eqref{Weak-N1} and \eqref{Weak-N2}, we conclude \eqref{Crucial-03}.

\subsection{Completion of the proof of Theorem~\ref{Thm-01}}

 Recall the size $\varepsilon$ of the initial data defined in \eqref{Size-of-initial-data}. By \eqref{Crucial-01} and \eqref{Crucial-03}, there exists a constant $C>0$, independent of $T$, such that $\ml{N}$ maps $B_R(T)$ into itself and satisfies
\begin{align}\label{Small-contraction-constant}
	d_T\big(\ml{N}[w],\ml{N}[\tilde w]\big)\leqslant\frac12\,d_T(w,\tilde w)
\end{align}
for all $w,\tilde w\in B_R(T)$, provided that
\begin{align*}
	R:=2C\varepsilon\ \ \mbox{and}\ \ C(R^{p-1}+R^{q-1})\leqslant\frac12.
\end{align*}

We define the Picard sequence by
\begin{align*}
	w^{(0)}:=(0,0)\ \ \mbox{and}\ \ w^{(k+1)}:=\ml{N}[w^{(k)}]
\end{align*}
for $k\in\mb{N}_0:=\mb{N}\cup\{0\}$. The self-map property shows inductively that $w^{(k)}\in B_R(T)$ for all $k\in\mb{N}_0$. Moreover, \eqref{Small-contraction-constant} gives
\begin{align*}
	d_T\big(w^{(k+1)},w^{(k)}\big)\leqslant\frac12\,d_T\big(w^{(k)},w^{(k-1)}\big)
\end{align*}
for every $k\in\mb{N}$. Hence, $\{w^{(k)}\}_{k\in\mb{N}_0}$ is a Cauchy sequence in $\mb{X}_T\times\mb{Y}_T$. By the completeness established in Section~\ref{Subsection-Weak-Contraction}, there exists $w=(u,v)\in\mb{X}_T\times\mb{Y}_T$ such that
\begin{align}\label{Picard-weak-convergence}
	d_T\big(w^{(k)},w\big)\to0\ \ \mbox{as}\ \ k\to\infty.
\end{align}

The convergence in \eqref{Picard-weak-convergence} implies
\begin{align*}
	u^{(k)}\to u\ \ &\mbox{in}\ \ \ml{C}\big([0,T],W^{1,r_n}\cap L^\infty\cap H^1\big),\\
	u_t^{(k)}\to u_t\ \ &\mbox{in}\ \ \ml{C}\big([0,T],L^2\big),\\
	v^{(k)}\to v\ \ &\mbox{in}\ \ \ml{C}\big([0,T],L^\infty\cap H^1\big),\\
	v_t^{(k)}\to v_t\ \ &\mbox{in}\ \ \ml{C}\big([0,T],L^2\big).
\end{align*}
It remains to recover the higher-order spatial regularity and the
corresponding time-weighted bounds. Because $w^{(k)}\in B_R(T)$ for every $k\in\mb{N}_0$, the sequences
\begin{align*}
	\left\{\langle t\rangle^{\frac{n}{2}\left(\frac{1}{m}-\frac{1}{r_n}\right)}u^{(k)}\right\}_{k\in\mb{N}_0},
	\ \ 
	\big\{u^{(k)}\big\}_{k\in\mb{N}_0},
	\ \ 
	\big\{\langle t\rangle u_t^{(k)}\big\}_{k\in\mb{N}_0},
\end{align*}
are bounded respectively in $L^\infty([0,T],W^{2,r_n})$, $L^\infty([0,T],H^2)$, $L^\infty([0,T],H^1)$, whereas
\begin{align*}
	\left\{[\Lambda_n(t)]^{-1}v^{(k)}\right\}_{k\in\mb{N}_0},
	\ \ 
	\big\{v_t^{(k)}\big\}_{k\in\mb{N}_0}
\end{align*}
are bounded respectively in $L^\infty([0,T],H^2)$ and $L^\infty([0,T],H^1)$. Recalling $1<r_n<\infty$, all the spatial Sobolev spaces appearing above are reflexive. Hence, by the vector-valued Banach--Alaoglu theorem and a diagonal extraction, there exist functions $U_1,U_2,U_3,V_1,V_2$ such that, along a subsequence,
\begin{align}
	\langle t\rangle^{\frac{n}{2}\left(\frac{1}{m}-\frac{1}{r_n}\right)}u^{(k)}\overset{\ast}{\rightharpoonup} U_1\ \ &\mbox{in}\ \ L^\infty\big([0,T],W^{2,r_n}\big),\label{Weak-star-limit-01}\\
	u^{(k)}\overset{\ast}{\rightharpoonup} U_2\ \ &\mbox{in}\ \ L^\infty\big([0,T],H^2\big),\label{Weak-star-limit-02}\\
	\langle t\rangle u_t^{(k)}\overset{\ast}{\rightharpoonup} U_3\ \ &\mbox{in}\ \ L^\infty\big([0,T],H^1\big),\label{Weak-star-limit-03}\\
	[\Lambda_n(t)]^{-1}v^{(k)}\overset{\ast}{\rightharpoonup} V_1\ \ &\mbox{in}\ \ L^\infty\big([0,T],H^2\big),\label{Weak-star-limit-04}\\
	v_t^{(k)}\overset{\ast}{\rightharpoonup} V_2\ \ &\mbox{in}\ \ L^\infty\big([0,T],H^1\big).\label{Weak-star-limit-05}
\end{align}
On the other hand, the strong convergence in \eqref{Picard-weak-convergence} implies convergence of the same sequences in the sense of distributions on $(0,T)\times\mb{R}^n$. By the uniqueness of distributional limits, we obtain
\begin{align*}
	U_1&=\langle t\rangle^{\frac{n}{2}\left(\frac{1}{m}-\frac{1}{r_n}\right)}
	u,&U_2&=u,&U_3&=\langle t\rangle u_t,\\
	V_1&=[\Lambda_n(t)]^{-1}v,&V_2&=v_t.
\end{align*}
The weak-$\ast$ lower semicontinuity of the corresponding
$L^\infty$ norms therefore gives
\begin{align*}
	&\left\|\langle t\rangle^{\frac{n}{2}\left(\frac{1}{m}-\frac{1}{r_n}\right)}u\right\|_{L^\infty([0,T],W^{2,r_n})}+\|u\|_{L^\infty([0,T],H^2)}+\|\langle t\rangle u_t\|_{L^\infty([0,T],H^1)}\\
	&+\big\|[\Lambda_n(t)]^{-1}v\big\|_{L^\infty([0,T],H^2)}+\|v_t\|_{L^\infty([0,T],H^1)}\\
	&\leqslant R.
\end{align*}
The remaining weighted $L^\infty$ bounds for $u$ and $v$ pass directly to the limit through the strong convergence in \eqref{Picard-weak-convergence}. Consequently,
\begin{align}
	\mathop{\rm ess\,sup}_{t\in[0,T]}
	\Big(&\langle t\rangle^{\frac{n}{2}\left(\frac{1}{m}-\frac{1}{r_n}\right)}\|u(t,\cdot)\|_{W^{2,r_n}}+\|u(t,\cdot)\|_{H^2}+\langle t\rangle\|u_t(t,\cdot)\|_{H^1}\nonumber\\
	&+\langle t\rangle^{\frac{n-1}{2}}\|v(t,\cdot)\|_{L^\infty}+[\Lambda_n(t)]^{-1}\|v(t,\cdot)\|_{H^2}+\|v_t(t,\cdot)\|_{H^1}\Big)\leqslant R.\label{Strong-bound-for-limit}
\end{align}
The bound \eqref{Strong-bound-for-limit} implies that the nonlinear compositions $|v|^p$ and $|u|^q$ are well-defined for almost every $t\in[0,T]$ and satisfy the estimates of Lemma~4.1 with the same constant $R$. Thus, the Duhamel operator $\ml{N}[w]$ is well-defined. The proof of the contraction estimate depends only on the bounds in \eqref{Strong-bound-for-limit}. Consequently, it extends to the weak closure of $B_R(T)$ and yields
\begin{align*}
	d_T\big(\ml{N}[w^{(k)}],\ml{N}[w]\big)\leqslant\frac12\,d_T\big(w^{(k)},w\big)\to0.
\end{align*}
On the other hand, $\ml{N}[w^{(k)}]=w^{(k+1)}$, and \eqref{Picard-weak-convergence} gives
\begin{align*}
	d_T\big(w^{(k+1)},w\big)\to0.
\end{align*}
Consequently, $d_T(\ml{N}[w],w)=0$, and hence $w=\ml{N}[w]$ in the weak evolution space. Applying the linear estimates to this Duhamel representation and using Lemma~\ref{Lemma-Nonlinear-Strong}, we obtain $\ml{N}[w]\in Z_T$ with $\|\ml{N}[w]\|_{Z_T}\leqslant R$. Therefore, $w\in B_R(T)$. In particular, the continuity of the linear solution operators and the Bochner integrals in the Duhamel formula gives the strong time continuity stated in Theorem~\ref{Thm-01}.

To prove uniqueness, let $w$ and $\tilde w$ be two fixed points in $B_R(T)$. By \eqref{Small-contraction-constant}, we conclude
\begin{align*}
	d_T(w,\tilde w)=d_T\big(\ml{N}[w],\ml{N}[\tilde w]\big)\leqslant\frac12\,d_T(w,\tilde w).
\end{align*}
Thus, $d_T(w,\tilde w)=0$, and therefore $w=\tilde w$. Finally, all constants are independent of $T$. The fixed points constructed on finite intervals are compatible by uniqueness and define a unique global in-time mild Sobolev solution on $[0,\infty)$. The estimates in Theorem~\ref{Thm-01} follow from the uniform bound $\|w\|_{Z_T}\leqslant R$. This completes the proof of Theorem~\ref{Thm-01}.

\subsection{Proof of Corollary~\ref{Cor-Free-Wave}}

 The key point is that the nonlinear source $|u|^q$ is integrable in the Sobolev norms propagated by the free wave group. In two space dimensions, the additional factor $\Lambda_2$ caused by the low-frequency growth of $W_1(t,|\nabla|)$ remains integrable because the polynomial decay exponent is strictly larger than $1$. 

Since the estimates in Theorem~\ref{Thm-01} are uniform with respect to the final time, we may define
\begin{align*}
	\|u\|_{X_\infty}:=\sup_{T>0}\|u\|_{X_T}<\infty.
\end{align*}
Then, by \eqref{Nonlinear-u-source}, for every $\tau\geqslant0$, we have
\begin{align}\label{Scattering-source-estimate}
	\|\,|u(\tau,\cdot)|^q\|_{H^1\cap L^1}\lesssim\langle\tau\rangle^{-q\frac n2\left(\frac1m-\frac1{r_n}\right)}\|u\|_{X_\infty}^q.
\end{align}
Recall from \eqref{Time-integrability-condition} that $q\frac n2\left(\frac1m-\frac1{r_n}\right)>1$. 
Therefore,
\begin{align}\label{Weighted-source-integrability}
	\int_0^\infty\Lambda_n(\tau)\|\,|u(\tau,\cdot)|^q\|_{H^1\cap L^1}\,\mathrm{d}\tau&\lesssim\|u\|_{X_\infty}^q\int_0^\infty\Lambda_n(\tau)\langle\tau\rangle^{-q\frac n2\left(\frac1m-\frac1{r_n}\right)}\,\mathrm{d}\tau<\infty.
\end{align}
Indeed, $\Lambda_3(\tau)=1$, while $\Lambda_2(\tau)=\sqrt{\ln(\mathrm{e}+\tau)}$.

By \eqref{W1-energy-bound} with $s_1=1$, it holds that
\begin{align}\label{Scattering-W1-bound}
	\|W_1(\tau,|\nabla|)|u(\tau,\cdot)|^q\|_{H^2}\lesssim\Lambda_n(\tau)\|\,|u(\tau,\cdot)|^q\|_{H^1\cap L^1}.
\end{align}
Moreover, the boundedness of $W_0(\tau,|\nabla|)$ on $H^1$ gives
\begin{align}\label{Scattering-W0-bound}
	\|W_0(\tau,|\nabla|)|u(\tau,\cdot)|^q\|_{H^1}\lesssim\|\,|u(\tau,\cdot)|^q\|_{H^1}.
\end{align}
The solution regularity in Theorem~\ref{Thm-01} and the composition estimates in Lemma~\ref{Lemma-Nonlinear-Strong} also imply the strong measurability of the two integrands. Hence, \eqref{Weighted-source-integrability}, \eqref{Scattering-W1-bound} and \eqref{Scattering-W0-bound} show that the following Bochner integrals converge absolutely in $H^2$ and $H^1$, respectively:
\begin{align}\label{Scattering-data}
	v_0^+(x)&:=v_0(x)-\int_0^\infty W_1(\tau,|\nabla|)|u(\tau,x)|^q\,\mathrm{d}\tau,\notag\\
	v_1^+(x)&:=v_1(x)+\int_0^\infty W_0(\tau,|\nabla|)|u(\tau,x)|^q\,\mathrm{d}\tau.
\end{align}

We define the free wave associated with these asymptotic data by
\begin{align}\label{Definition-v-plus}
	v^+(t,x):=W_0(t,|\nabla|)v_0^+(x)+W_1(t,|\nabla|)v_1^+(x).
\end{align}
For every $t,\tau\in\mb{R}$, the trigonometric identities for the Fourier multipliers yield
\begin{align}\label{Wave-group-identities}
	-W_0(t,|\nabla|)W_1(\tau,|\nabla|)+W_1(t,|\nabla|)W_0(\tau,|\nabla|)&=W_1(t-\tau,|\nabla|),\notag\\
	|\nabla|^2W_1(t,|\nabla|)W_1(\tau,|\nabla|)+W_0(t,|\nabla|)W_0(\tau,|\nabla|)&=W_0(t-\tau,|\nabla|).
\end{align}
Using \eqref{Scattering-data}, \eqref{Definition-v-plus} and the first identity in \eqref{Wave-group-identities}, we obtain
\begin{align*}
	v^+(t,x)=W_0(t,|\nabla|)v_0(x)+W_1(t,|\nabla|)v_1(x) +\int_0^\infty W_1(t-\tau,|\nabla|)|u(\tau,x)|^q\,\mathrm{d}\tau.
\end{align*}
On the other hand, the Duhamel formula for the undamped component reads
\begin{align*}
	v(t,x)=W_0(t,|\nabla|)v_0(x)+W_1(t,|\nabla|)v_1(x)+\int_0^tW_1(t-\tau,|\nabla|)|u(\tau,x)|^q\,\mathrm{d}\tau.
\end{align*}
Subtracting the preceding two identities gives
\begin{align}\label{Scattering-tail-v}
	v(t,x)-v^+(t,x)=-\int_t^\infty W_1(t-\tau,|\nabla|)|u(\tau,x)|^q\,\mathrm{d}\tau.
\end{align}
Since $\partial_tW_0(t,|\nabla|)=-|\nabla|^2W_1(t,|\nabla|)$ and $\partial_tW_1(t,|\nabla|)=W_0(t,|\nabla|)$, the second identity in \eqref{Wave-group-identities} similarly yields
\begin{align}\label{Scattering-tail-vt}
	v_t(t,x)-v_t^+(t,x)=-\int_t^\infty W_0(t-\tau,|\nabla|)|u(\tau,x)|^q\,\mathrm{d}\tau.
\end{align}

Notice that
\begin{align*}
	W_0(-s,|\nabla|)=W_0(s,|\nabla|)\ \ \mbox{and}\ \ W_1(-s,|\nabla|)=-W_1(s,|\nabla|).
\end{align*}
Therefore, the estimates in Lemma~\ref{Lemma-Wave} remain valid for negative times after replacing the time variable by its absolute value. Applying \eqref{W1-energy-bound} to \eqref{Scattering-tail-v}, and using $\Lambda_n(\tau-t)\leqslant\Lambda_n(\tau)$ for $\tau\geqslant t$, we derive
\begin{align*}
	\|v(t,\cdot)-v^+(t,\cdot)\|_{H^2}&\lesssim\int_t^\infty\Lambda_n(\tau-t)\|\,|u(\tau,\cdot)|^q\|_{H^1\cap L^1}\,\mathrm{d}\tau\\
	&\lesssim\int_t^\infty\Lambda_n(\tau)\|\,|u(\tau,\cdot)|^q\|_{H^1\cap L^1}\,\mathrm{d}\tau\to0
\end{align*}
as $t\to\infty$, because the last expression is the tail of the convergent integral in \eqref{Weighted-source-integrability}. Likewise, the boundedness of $W_0(t,|\nabla|)$ on $H^1$ and \eqref{Scattering-tail-vt} imply
\begin{align*}
	\|v_t(t,\cdot)-v_t^+(t,\cdot)\|_{H^1}&\lesssim\int_t^\infty\|\,|u(\tau,\cdot)|^q\|_{H^1}\,\mathrm{d}\tau\\
	&\lesssim\int_t^\infty\Lambda_n(\tau)\|\,|u(\tau,\cdot)|^q\|_{H^1\cap L^1}\,\mathrm{d}\tau\to0
\end{align*}
as $t\to\infty$. In view of \eqref{Definition-v-plus}, these two limits are exactly the convergence stated in Corollary~\ref{Cor-Free-Wave}. Hence, the proof is complete.

\section*{Acknowledgments}
Wenhui Chen is supported in part by the National Natural Science Foundation of China (grant No. 12301270) and the Guangdong Basic and Applied Basic Research Foundation (grant No. 2025A1515010240). The author would like to thank Kosuke Kita for kindly sharing the recent preprint \cite{Georgiev-Kita=2026}.


\begin{thebibliography}{99}
\bibitem{Agemi-Kurokawa-Takamura=2000}
\newblock R. Agemi, Y. Kurokawa, H. Takamura.
\newblock Critical curve for $p-q$ systems of nonlinear wave equations in three space dimensions.
\newblock \emph{J. Differential Equations} \textbf{167} (2000), no. 1, 87--133.

\bibitem{Chen=2022}
\newblock W. Chen.
\newblock Blow-up and lifespan estimates for Nakao's type problem with nonlinearities of derivative type.
\newblock \emph{Math. Methods Appl. Sci.} \textbf{45} (2022), no. 10, 5988--6004.

\bibitem{Chen-Dao=2023}
\newblock W. Chen, T. A. Dao.
\newblock Sharp lifespan estimates for the weakly coupled system of semilinear damped wave equations in the critical case.
\newblock \emph{Math. Ann.} \textbf{385} (2023), no. 1-2, 101--130.

\bibitem{Chen-Ikehata=2026}
\newblock W. Chen, R. Ikehata.
\newblock Large time behavior for the classical wave equation with different regular data and its applications.
\emph{Asymptot. Anal.}, to appear, DOI: 10.1177/09217134261440139.

\bibitem{Chen-Palmieri=2020}
\newblock W. Chen, A. Palmieri.
\newblock Nonexistence of global solutions for the semilinear Moore-Gibson-Thompson equation in the conservative case.
\newblock \emph{Discrete Contin. Dyn. Syst.} \textbf{40} (2020), no. 9, 5513--5540.

\bibitem{Chen-Palmieri=2026}
\newblock W. Chen, A. Palmieri.
\newblock Blow-up and sharp lifespan estimates for a weakly coupled system of semilinear wave equations on a compact Lie group.
\newblock Preprint, arXiv:2604.06626.

\bibitem{Chen-Takeda=2023}
\newblock W. Chen, H. Takeda.
\newblock Large-time asymptotic behavior for the classical thermoelastic system.
\newblock \emph{J. Differential Equations} \textbf{377} (2023), 809--848.

\bibitem{Chen-Reissig=2021}
\newblock W. Chen, M. Reissig.
\newblock Blow-up of solutions to Nakao's problem via an iteration argument.
\newblock \emph{J. Differential Equations} \textbf{275} (2021), 733--756.

\bibitem{Delsanto=1997}
\newblock D. Del Santo.
\newblock Global existence and blow-up for a hyperbolic system in three space dimensions.
\newblock \emph{Rend. Istit. Mat. Univ. Trieste} \textbf{29} (1997), no. 1-2, 115--140.

\bibitem{DelSanto-Georgiev-Mitidieri=1997}
\newblock D. Del Santo, V. Georgiev, E. Mitidieri.
\newblock Global existence of the solutions and formation of singularities for a class of hyperbolic systems.
\newblock \emph{Progr. Nonlinear Differential Equations Appl.}, \textbf{32}
Birkh\"auser, Boston, Inc., Boston, MA, 1997, 117--140.

\bibitem{DelSanto-Mitidieri=1998}
\newblock D. Del Santo, E. Mitidieri.
\newblock Blow-up of solutions of a hyperbolic system: the critical case.
\newblock \emph{Differ. Uravn.} \textbf{34} (1998), no. 9, 1155--1161, 1293. Translation in
\emph{Differential Equations} \textbf{34} (1998), no. 9, 1157--1163.

\bibitem{Fujita=1966}
\newblock H. Fujita.
\newblock On the blowing up of solutions of the Cauchy problem for $u_t=\Delta u+u^{1+\alpha}$.
\newblock \emph{J. Fac. Sci. Univ. Tokyo Sect. I} \textbf{13} (1966), 109--124.

\bibitem{Georgiev-Kita=2026}
\newblock V. Georgiev, K. Kita.
\newblock Weighted pointwise estimates for the damped wave equation and global solutions to Nakao’s problem in three dimensions.
\newblock Preprint.

\bibitem{Georgiev-Takamura-Zhou=2006}
\newblock V. Georgiev, H. Takamura, Y. Zhou.
\newblock The lifespan of solutions to nonlinear systems of a high-dimensional wave equation.
\newblock \emph{Nonlinear Anal.} \textbf{64} (2006), no. 10, 2215--2250.

\bibitem{Ikeda-Inui-Okamoto-Wakasugi=2019}
\newblock M. Ikeda, T. Inui, M. Okamoto, Y. Wakasugi.
\newblock $L^p$-$L^q$ estimates for the damped wave equation and the critical exponent for the nonlinear problem with slowly decaying data.
\newblock \emph{Commun. Pure Appl. Anal.} \textbf{18} (2019), no. 4, 1967--2008.

\bibitem{Ikeda-Sobajima-Wakasa=2019}
\newblock M. Ikeda, M. Sobajima, K. Wakasa.
\newblock Blow-up phenomena of semilinear wave equations and their weakly coupled systems.
\newblock \emph{J. Differential Equations} \textbf{267} (2019), no. 9, 5165--5201.

\bibitem{Ikehata=2023}
\newblock R. Ikehata.
\newblock $L^2$-blowup estimates of the wave equation and its application to local energy decay.
\newblock \emph{J. Hyperbolic Differ. Equ.} \textbf{20} (2023), no. 1, 259--275.

\bibitem{John=1979}
\newblock F. John.
\newblock Blow-up of solutions of nonlinear wave equations in three space dimensions.
\newblock \emph{Manuscripta Math.} \textbf{28} (1979), no. 1-3, 235--268.

\bibitem{John=1981}
\newblock F. John.
\newblock \emph{Plane Waves and Spherical Means Applied to Partial Differential Equations}.
\newblock Springer-Verlag, New York-Berlin, 1981.

\bibitem{Kita-Kusaba=2022}
\newblock K. Kita, R. Kusaba.
\newblock A remark on the blowing up of solutions to Nakao's problem.
\newblock \emph{J. Math. Anal. Appl.} \textbf{513} (2022), no. 1, Paper No. 126199, 20 pp.

\bibitem{Kubo-Ohta=1999}
\newblock H. Kubo, M. Ohta.
\newblock Critical blowup for systems of semilinear wave equations in low space dimensions.
\newblock \emph{J. Math. Anal. Appl.} \textbf{240} (1999), no. 2, 340--360.

\bibitem{Kurokawa=2005}
\newblock Y. Kurokawa.
\newblock The lifespan of radially symmetric solutions to nonlinear systems of odd dimensional wave equations.
\newblock \emph{Nonlinear Anal.} \textbf{60} (2005), no. 7, 1239--1275.

\bibitem{Kurokawa-Takamura=2003}
\newblock Y. Kurokawa, H. Takamura.
\newblock A weighted pointwise estimate for two dimensional wave equations and its applications to nonlinear systems.
\newblock \emph{Tsukuba J. Math.} \textbf{27} (2003), no. 2, 417--448.

\bibitem{Kurokawa-Takamura-Wakasa=2012}
\newblock Y. Kurokawa, H. Takamura, K. Wakasa.
\newblock The blow-up and lifespan of solutions to systems of semilinear wave equation with critical exponents in high dimensions.
\newblock \emph{Differential Integral Equations} \textbf{25} (2012), no. 3-4, 363--382.

\bibitem{Li-Zhou=1995}
\newblock T. T. Li, Y. Zhou.
\newblock Breakdown of solutions to $\square u+u_t=|u|^{1+\alpha}$.
\newblock \emph{Discrete Contin. Dynam. Systems} \textbf{1} (1995), no. 4, 503--520.

\bibitem{Li-Palmieri=2025}
\newblock Y. Li, A. Palmieri.
\newblock On the blow-up of solutions to a Nakao-type problem with a time-dependent damping term.
\newblock Preprint, arXiv:2510.17368.

\bibitem{Li-Palmieri2=2025}
\newblock Y. Li, A. Palmieri.
\newblock Blow-up results for a Nakao-type problem with a time-dependent damping term and derivative-type nonlinearities.
\newblock Preprint, arXiv:2510.18378.

\bibitem{Liu=2026}
\newblock M. Liu.
\newblock Quantitative blow-up via renormalized Kato theory: Resolving Nakao-type systems.
\newblock \emph{J. Differential Equations} \textbf{462} (2026), Paper No. 114165, 14 pp.

\bibitem{Nakao=2016}
\newblock M. Nakao.
\newblock Global existence to the initial-boundary value problem for a system of semilinear wave equations.
\newblock \emph{Nonlinear Anal.} \textbf{146} (2016), 233--257.

\bibitem{Nakao=2018}
\newblock M. Nakao.
\newblock Global existence to the initial-boundary value problem for a system of nonlinear diffusion and wave equations.
\newblock \emph{J. Differential Equations} \textbf{264} (2018), no. 1, 134--162.

\bibitem{Narazaki=2009}
\newblock T. Narazaki.
\newblock Global solutions to the Cauchy problem for the weakly coupled system of damped wave equations.
\newblock \emph{Discrete Contin. Dyn. Syst.} 2009, suppl., Dynamical systems, differential equations and applications. 7th AIMS Conference, 592--601.

\bibitem{Nishihara=2003}
\newblock K. Nishihara.
\newblock $L^p$-$L^q$ estimates of solutions to the damped wave equation in 3-dimensional space and their application.
\newblock \emph{Math. Z.} \textbf{244} (2003), no. 3, 631--649.

\bibitem{Nishihara=2012}
\newblock K. Nishihara.
\newblock Asymptotic behavior of solutions for a system of semilinear heat equations and the corresponding damped wave system.
\newblock \emph{Osaka J. Math.} \textbf{49} (2012), no. 2, 331--348.

\bibitem{Nishihara-Wakasugi=2014}
\newblock K. Nishihara, Y. Wakasugi.
\newblock Critical exponent for the Cauchy problem to the weakly coupled damped wave system.
\newblock \emph{Nonlinear Anal.} \textbf{108} (2014), 249--259.

\bibitem{Nishihara-Wakasugi=2015}
\newblock K. Nishihara, Y. Wakasugi.
\newblock Global existence of solutions for a weakly coupled system of semilinear damped wave equations.
\newblock \emph{J. Differential Equations} \textbf{259} (2015), no. 8, 4172--4201.

\bibitem{Palmieri-Reissig=2018}
\newblock A. Palmieri, M. Reissig.
\newblock Semi-linear wave models with power non-linearity and scale-invariant time-dependent mass and dissipation, II.
\newblock \emph{Math. Nachr.} \textbf{291} (2018), no. 11-12, 1859--1892.

\bibitem{Palmieri-Takamura=2023}
\newblock A. Palmieri, H. Takamura.
\newblock A blow-up result for a Nakao-type weakly coupled system with nonlinearities of derivative-type.
\newblock \emph{Math. Ann.} \textbf{387} (2023), no. 1-2, 111--132.

\bibitem{Strauss=1981}
\newblock W. A. Strauss.
\newblock Nonlinear scattering theory at low energy.
\newblock \emph{J. Functional Analysis} \textbf{41} (1981), no. 1, 110--133.

\bibitem{Sun-Wang=2007}
\newblock F. Sun, M. Wang.
\newblock Existence and nonexistence of global solutions for a nonlinear hyperbolic system with damping.
\newblock \emph{Nonlinear Anal.} \textbf{66} (2007), no. 12, 2889--2910.

\bibitem{Takeda=2026}
\newblock H. Takeda.
\newblock $L^2$-estimates for the linear elastic waves.
\newblock \emph{Math. Ann.} \textbf{394} (2026), no. 4, Paper No. 82, 31 pp.

\bibitem{Wakasugi=2017}
\newblock Y. Wakasugi.
\newblock A note on the blow-up of solutions to Nakao's problem.
\newblock \emph{Trends Math. Res. Perspect.} Birkh\"auser/Springer, Cham, 2017, 545--551.

\bibitem{Zhang=2001}
\newblock Q. S. Zhang.
\newblock A blow-up result for a nonlinear wave equation with damping: the critical case.
\newblock \emph{C. R. Acad. Sci. Paris S\'er. I Math.} \textbf{333} (2001), no. 2, 109--114.
\end{thebibliography}
\end{document}